\documentclass[11pt]{article}
\usepackage{amssymb}
\usepackage{amsmath}
\usepackage{amsmath}
\usepackage{amsthm}
\usepackage{amsfonts}
\usepackage{graphicx}
\usepackage{enumerate}
\usepackage{comment}
\usepackage{bm}
\usepackage{booktabs}
\usepackage{bbm} 
\usepackage{dsfont}
\usepackage{natbib}
\usepackage{har2nat}
\usepackage{caption}
\usepackage{float}
\usepackage{xcolor,colortbl}
\usepackage{multirow}
\usepackage{multicol}
\usepackage{tikz}
\numberwithin{equation}{section}
\newtheorem{theorem}{Theorem}[section]

\newtheorem{proposition}{Proposition}[section]

\newtheorem{lemma}{Lemma}[section]
\theoremstyle{definition}
\newtheorem{assumption}{Assumption}[section]

\DeclareMathOperator{\var}{Var}

\newcommand{\norm}[1]{\left \|#1\right \|}
\newcommand{\abs}[1]{\left\vert #1 \right\vert}

\renewcommand{\hat}{\widehat}
\renewcommand{\tilde}{\widetilde}

\newcommand{\Var}[0]{\mathrm{Var}}

\newcommand{\argmin}{\text argmin}

\renewcommand{\argmin}{\rm argmin}

\renewcommand{\hat}{\widehat}
\renewcommand{\bar}{\overline}
\allowdisplaybreaks

\begin{document}
\title{Extremal Quantiles under Two-Way Clustering\thanks{We are grateful to Antonio F. Galvao, Bryan S. Graham, Michael Jansson, and Yassine Sbai Sassi for their invaluable feedback and suggestions. We also thank the participants of the SEA 2024 conference and NTU CRETA seminar for their comments and discussions. Sasaki thanks Brian and Charlotte Grove Chair for research support. All the remaining errors are ours.\bigskip}}
\author{
	Harold D. Chiang\thanks{Harold D. Chiang: hdchiang@wisc.edu. Department of Economics, University of Wisconsin-Madison, William H. Sewell Social Science Building, 1180 Observatory Drive,	Madison, WI 53706-1393, USA\medskip} 
	\qquad 
	Ryutah Kato\thanks{Ryutah Kato: kato-ryuta-wh@ynu.ac.jp. Graduate School of International Social Sciences, Yokohama National University\medskip} 
	\qquad 
	Yuya Sasaki\thanks{Yuya Sasaki: yuya.sasaki@vanderbilt.edu. Department of Economics, Vanderbilt University, VU Station B \#351819, 2301 Vanderbilt Place, Nashville, TN 37235-1819, USA\medskip}
}
\date{}
\maketitle

\begin{abstract}
This paper studies extremal quantiles under two-way clustered dependence. 
We show that the limiting distribution of unconditional intermediate-order tail quantiles is Gaussian. 
This result is notable because two-way clustering typically leads to non-Gaussian limiting behavior. 
Remarkably, extremal quantiles remain asymptotically Gaussian even in degenerate cases. 
Building on this insight, we extend our analysis to extremal quantile regression at intermediate orders. 
Simulation results corroborate our theoretical findings. 
Finally, we provide an empirical application to growth-at-risk, showing that earlier empirical conclusions remain robust even after accounting for two-way clustered dependence in panel data and the focus on extreme quantiles.
\end{abstract}

\newpage

\section{Introduction}
Quantile regression has become increasingly important in economics and related fields for analyzing the effects of covariates across different parts of the conditional distribution. Many key applications focus on the study of rare events, or quantiles in tails. Notable examples include value-at-risk analysis in financial economics \citep{Chernozhukov2001, Engle2004}, the analysis of highest bids in auctions \citep{chernozhukov2004likelihood, Hirano2003}, the investigation of birth weights among live infants \citep*{Chernozhukov2011, Zhang2018, Sasaki2022, Sasaki2023}, and the growth-at-risk (GaR) and related analyses \citep[e.g.,][]{adrian2019vulnerable, adrian2022term, kiley2022unemployment, kiley2024growth, lopez2024inflation}, among others.

The conventional large-sample theory for quantile regression encounters difficulties in the extreme tails of the distribution due to data sparsity. As a result, traditional inference tools, such as the analytical variance estimator (\citealt{koenker2005quantile}, Section 3.4) and the pivotal bootstrap \citep*{chernozhukov2009finite}, are not directly applicable in these regions. However, advances in modern extreme value theory have helped overcome this limitation for unconditional quantiles \citep*{Leadbetter1983, Resnick1987, Embrechts1997}. \citet{Chernozhukov2005} extends this theoretical framework to estimation and inference for extremal conditional quantiles. Further developments include asymptotic theory and inference methods for quantile treatment effects in the tails, as explored by \citet*{d2018extremal}, \citet{Zhang2018}, \citet*{Deuber2021}, and \citet*{Sasaki2023}.

At the same time, applied economic research increasingly involves data exhibiting complex clustering structures. In particular, two-way clustering is frequently encountered in empirical work. Common examples include panel data, matched employer–employee data, matched student–teacher data, scanner data indexed by stores and products, market share data indexed by markets and products, and growth or development data indexed by ethnicity and geographic units. These examples underscore the prevalence of multi-way clustering in real-world datasets. It is essential to account for the complications that arise under two-way clustering in order to conduct reliable and valid econometric analysis.

The seminal work by \citet*{Cameron2011} introduces methods for two-way cluster-robust inference, which has had a substantial influence on the field. \citet{Menzel2021} provides a formal analysis of the asymptotic properties and bootstrap validity in the presence of two-way clustering. \citet*{Davezies2021} advance the development of empirical process theory for such settings, and \citet*{MacKinnon2021} propose wild bootstrap procedures that explicitly account for clustering effects. In a related direction, \citet{Graham2020} studies sparse network asymptotics, with a focus on logistic regression models.

Our work builds on these foundational contributions and focuses on the estimation of rare events in data sampled under two-way clustering. This focus is not only novel and empirically relevant, but also theoretically crucial, as it addresses the distinctive challenges and provides new insights. Two-way clustering introduces dependence structures that complicate statistical analysis, particularly in distributional tails. By targeting tail estimation, we deepen understanding of the behavior of extreme observations in such settings.
Notably, while \citet{Menzel2021} shows that the asymptotic distribution of sample means and related statistics may deviate from Gaussianity under two-way clustering, we find that extremal quantiles exhibit robustness to such deviations, even in degenerate cases.

The techniques underlying our asymptotic theory are closely related to prior research developed under alternative asymptotic frameworks. Our approach draws parallels with the analysis of network subgraph moments by \citet*{Bickel2011}, dyadic density estimation by \citet*{Graham2022, chiang2023empirical, cattaneo2023uniform}, and the semiparametric asymptotic regimes studied by \citet*{cattaneo2014small, cattaneo2018alternative}, in which both linear and quadratic components play essential roles. These studies provide valuable insights and foundational frameworks upon which our work builds.
In addition, this paper contributes to the growing literature on inference for quantile regression with cluster-sampled data, including \citet{hagemann2017cluster, hagemann2024inference}, and connects with the asymptotic theory for various regression models in dyadic data contexts, such as \citet{graham2020dyadic} and \citet{sassi2023ordinary}.

The remainder of the paper is organized as follows. Section~2 analyzes unconditional quantiles under two-way clustering, presenting both the theoretical framework and supporting simulation results. Section~3 extends these results to quantile regression. Section~4 provides an empirical application to growth-at-risk (GaR). Section~5 concludes. All mathematical proofs are collected in the appendices.

\section{Extremal Quantiles}\label{sec:unconditional}
\subsection{The Setup and Estimation}\label{sec:setup}
We begin our analysis by examining the limiting distribution of unconditional quantiles in the tails. In the subsequent section, we extend this baseline framework to extremal quantile regression. The core ideas and methodological approach remain consistent across both settings. 
We first focus on unconditional quantiles because this setting allows us to more transparently illustrate how Gaussianity is achieved.

Consider a random variable $Y$ taking values in $\mathbb{R}$, and let its distribution function be denoted by $F_Y$. 
Our focus is on the extremal quantile
$$
F_Y^{-1}(\tau) = \inf \left\{ y : F_Y(y) > \tau \right\}
$$
for small values of $\tau$ (close to $0$). A symmetric argument applies to the case when $\tau$ is close to $1$.

Suppose we observe data $\{Y_{it} : 1 \le i \le N,\ 1 \le t \le T\}$, where $Y_{it} \overset{d}{=} Y$ is generated according to
$$
Y_{it} = f(\alpha_i, \gamma_t, \varepsilon_{it})
$$
for some Borel measurable function $f$, where $\alpha_1, \ldots, \alpha_N$, $\gamma_1, \ldots, \gamma_T$, and $\varepsilon_{11}, \ldots, \varepsilon_{NT}$ are mutually independent random vectors. 
Each of the latent components, $\alpha_i$, $\gamma_t$, and $\varepsilon_{it}$, is allowed to be vector-valued with possibly countably infinite dimensions. This setup accommodates dependence within each row $\left(Y_{i1}, \ldots, Y_{iT}\right)$ and within each column $\left(Y_{1t}, \ldots, Y_{Nt}\right)$ via the shared latent factors.

The unconditional $\tau$-th quantile $\beta(\tau)$ of $Y$ is estimated by the sample quantile
\begin{equation}\label{eq:def_beta_hat}
\hat{\beta}(\tau) \in \arg \min_{\beta \in \mathbb{R}} \sum_{i=1}^N \sum_{t=1}^T \rho_\tau\left(Y_{it} - \beta\right),
\end{equation}
where $\rho_\tau$ defined by $\rho_\tau(u) = u(\tau - \mathbbm{1}\{u < 0\})$ denotes the check function.

\subsection{Asymptotic Theory for the Extremal Quantiles}\label{sec:asymptotic_theory}
To conduct inference for $\beta(\tau)$ based on $\hat{\beta}(\tau)$, we study the limiting distribution of the statistic
$$
\widehat{Z}_{NT} = r_{NT} \left(\hat{\beta}(\tau) - \beta(\tau)\right),
$$
where $\{r_{NT}\}$ is a sequence that governs the rate of convergence. Our focus is on the asymptotic behavior of $\widehat{Z}_{NT}$, and consequently of $\hat{\beta}(\tau)$, in the left tail of the distribution of $Y$, as $\tau \to 0$.

For ease of writing, define
\begin{equation}\label{eq:w}
W_{NT}(\tau) \equiv \frac{-1}{\sqrt{\tau NT}} \sum_{i=1}^N \sum_{t=1}^T \left( \tau - \mathbbm{1}\left\{ Y_{it} < \beta(\tau) \right\} \right),
\end{equation}
and denote its variance by
$$
\sigma^2_{NT} = \Var \left( W_{NT}(\tau) \right).
$$
Let $m > 1$ be fixed. With these notations, write the statistic as
$$
\widehat{Z}_{NT} \equiv \frac{a_{NT}}{\sigma_{NT}} \left( \hat{\beta}(\tau) - \beta(\tau) \right), \quad \text{where} \quad a_{NT} \equiv \frac{\sqrt{\tau NT}}{\beta(m \tau) - \beta(\tau)}.
$$

By equation~\eqref{eq:def_beta_hat}, the transformation $\widehat{Z}_{NT}$ of $\hat{\beta}(\tau)$ minimizes the objective $Q_{NT}(\cdot,\tau)$, where
\begin{equation}\label{eq:q_original}
Q_{NT}(z, \tau) \equiv \frac{a_{NT}}{\sigma^2_{NT} \sqrt{\tau NT}} \sum_{i=1}^N \sum_{t=1}^T \left( \rho_\tau\left(Y_{it} - \beta(\tau) - \frac{\sigma_{NT} \cdot z}{a_{NT}}\right) - \rho_\tau\left(Y_{it} - \beta(\tau)\right) \right).
\end{equation}
To rewrite \eqref{eq:q_original} in a form more convenient for asymptotic analysis, we use Knight's identity:
$$
\rho_\tau(u - v) - \rho_\tau(u) = -v(\tau - \mathbbm{1}\{u < 0\}) + \int_0^v \left( \mathbbm{1}\{u \le s\} - \mathbbm{1}\{u \le 0\} \right) ds.
$$
Applying this identity, we can express the objective \eqref{eq:q_original} as
\begin{equation}\label{eq:q_decomposition}
Q_{NT}(z, \tau) = \sigma_{NT}^{-1} W_{NT}(\tau) \cdot z + G_{NT}(z, \tau),
\end{equation}
where $W_{NT}(\tau)$ is defined in \eqref{eq:w} and $G_{NT}(z, \tau)$ is given by
$$
G_{NT}(z, \tau) \equiv \frac{a_{NT}}{\sigma^2_{NT} \sqrt{\tau NT}} \sum_{i=1}^N \sum_{t=1}^T \int_0^{\sigma_{NT} \cdot z / a_{NT}} \left[ \mathbbm{1}\left(Y_{it} - \beta(\tau) \le s\right) - \mathbbm{1}\left(Y_{it} - \beta(\tau) \le 0\right) \right] ds.
$$

To analyze the asymptotic behavior of $\widehat{Z}_{NT}$, therefore, we first study the asymptotic properties of the components, $W_{NT}(\tau)$ and $G_{NT}(z, \tau)$, in the decomposition of the objective function $Q_{NT}(z, \tau)$ in \eqref{eq:q_decomposition}. To this end, we introduce a set of assumptions.

Recall that $F_Y$ denotes the distribution function of $Y$.
Assume, without loss of generality, by location normalization that the lower endpoint of its support is $s_Y = 0$ or $s_Y = -\infty$ for notational convenience. Suppose that $F_Y$ satisfies a domain-of-attraction condition of Type 1, 2, or 3, defined as follows.
 
\begin{enumerate}[${}$\quad Type 1: \ ]
\item
as $y \searrow s_Y=0$ or $-\infty$, 
$
F_Y(y+v a(y)) \sim F_Y(y) e^v \ \text{ for all } v \in \mathbb{R} \ (\xi \equiv 0);
$
\item
as $y \searrow s_Y=-\infty$, \hspace{7mm}
$
F_Y(v y) \sim v^{-1 / \xi} F_Y(y)  \ \text{ for all } v>0 \text{ for some } \xi>0; \qquad\text{and}
$
\item
as $y \searrow s_Y=0$, \hspace{12mm}
$
F_Y(v y) \sim v^{-1 / \xi} F_Y(y) \ \text{ for all } v>0 \text{ for some } \xi<0;
$
\end{enumerate}
where 
$a(y) \equiv \int_{s_Y}^y F_Y(v) d v / F_Y(y) \text { for } y>s_Y$.
This requirement is mild, as most common distribution families fall into one of the three types: Type 1 (e.g., normal), Type 2 (e.g., Student's $t$), or Type 3 (e.g., uniform). The parameter $\xi$ is referred to as the \emph{tail index}, which characterizes the heaviness of the left tail. Specifically, the tail is heavy if $\xi > 0$, light (or thin) if $\xi = 0$, and bounded if $\xi < 0$.

In addition, we make the following assumptions.

\begin{assumption}\label{a:regular}
 (i) $\tau \mapsto \displaystyle\frac{\partial F_Y^{-1}(\tau)}{\partial \tau}$
is regularly varying at 0 with exponent $-\xi-1$; and \\
(ii) Each of $\tau \mapsto \displaystyle f_Y( F_Y^{-1}(\tau) | \alpha_i)$, $\displaystyle f_{Y}( F_Y^{-1}(\tau) |\gamma_t )$, $\displaystyle f_{Y}( F_Y^{-1}(\tau)|\alpha_i,\gamma_t)$
is regularly varying at 0.
\end{assumption}

\begin{assumption}\label{a:b}
As $\tau\to 0$, $N, T \to \infty$ and $NT\tau\to \infty$, the following conditions hold:\\
(i) $E [P^2\left(Y_{it}<\beta(\tau)|\alpha_i\right)]=o(\tau\land\tau^2 N)$;
(ii) $E [P^2\left(Y_{it}<\beta(\tau)|\gamma_t\right)]=o(\tau\land\tau^2 T)$; and \\
(iii) $E [P^2\left(Y_{it}<\beta(\tau)|\alpha_i, \gamma_t\right)]=o(\tau)$.
\end{assumption}

\begin{assumption}\label{a:density}
As $y\to s_Y$, the following conditions hold: \\
(i) $f_Y(y|\alpha_i)=O(f_Y(y))$;
(ii) $f_Y(y|\gamma_t)=O(f_Y(y))$; and
(iii) $f_Y(y|\alpha_i,\gamma_t)=O(f_Y(y))$.
\end{assumption}

\begin{assumption}\label{a:lyapunov}
As $\tau\to 0$ and $N, T \to \infty$, the following conditions hold:\\
(i) $\displaystyle\frac{E[P\left(Y_{it}< \beta(\tau)\right) |\alpha_i)^3]}{E[P\left(Y_{it}< \beta(\tau)\right) |\alpha_i)^2]^{\frac{3}{2}}}=o(\sqrt{N})$; 
(ii) $\displaystyle\frac{E[P\left(Y_{it}< \beta(\tau)\right) |\gamma_t)^3]}{E[P\left(Y_{it}< \beta(\tau)\right) |\gamma_t)^2]^{\frac{3}{2}}}=o(\sqrt{T})$; and
 \\
(iii) $P(Y_{it}< \beta(\tau)|\alpha_i,\gamma_t)\cdot NT \to \infty$.
\end{assumption}

Assumption~\ref{a:regular} requires the existence and regular variation of the unconditional and conditional quantile density functions. 
The exponents may differ across 
$f_Y\left(F_Y^{-1}(\tau) \mid \alpha_i\right)$, 
$f_Y\left(F_Y^{-1}(\tau) \mid \gamma_t\right)$, and 
$f_Y\left(F_Y^{-1}(\tau) \mid \alpha_i, \gamma_t\right)$.
Assumption~\ref{a:b} places constraints on the conditional cumulative distribution function of $Y$. For instance, part~(i) requires that, as $\tau \to 0$, the second moment of the conditional probability 
$P\left(Y_{it} < \beta(\tau) \mid \alpha_i\right)$ 
diminishes at a rate faster than or comparable to $\tau$ (or $\tau^2 N$ if $N$ is small). In other words, the conditional probability must decay sufficiently fast as $\tau$ becomes small. 
The second moment of the unconditional probability 
$P(Y < \beta(\tau))$ 
scales with $\tau^2$, so the condition imposed by Assumption~\ref{a:b} is relatively mild and, in particular, weaker than requiring the conditional probability to decay at the same rate as the unconditional one.
Assumption~\ref{a:density} imposes regularity conditions on the relationship between the conditional and unconditional density functions of $Y$. For example, part~(i) asserts that the conditional density 
$f_Y(y \mid \alpha_i)$ 
should not vanish at a rate faster than the unconditional density $f_Y(y)$ as $y \to s_Y$. In essence, the decline of the conditional density in the lower tail should be comparable to that of the unconditional density.
Assumption~\ref{a:lyapunov} imposes moment conditions necessary to satisfy the Lyapunov condition, ensuring asymptotic Gaussianity of the normalized sample quantile.

Under these four assumptions, we can establish
\begin{align*}
\sigma_{NT}^{-1} W_{NT}(\tau) &\xrightarrow{d} \mathcal{N}(0, 1)
\qquad \text{and} \\
G_{NT}(z, \tau) &\xrightarrow{p} \frac{1}{2} \cdot \frac{m^{-\xi} - 1}{-\xi} \cdot z^2.
\end{align*}
Formal derivations are provided in Lemmas~\ref{lemma:w} and~\ref{lemma:g} in Appendices~\ref{sec:lemma:w} and~\ref{sec:lemma:g}, respectively.
As a consequence, we obtain the following limiting distribution for the solution $\widehat{Z}_{NT}$ to $\min_{z \in \mathbb{R}} Q_{NT}(z,\tau)$, where $Q_{NT}(z, \tau) = \sigma_{NT}^{-1} W_{NT}(\tau) \cdot z + G_{NT}(z, \tau)$ as defined in equation~\eqref{eq:q_decomposition}.

\begin{theorem}\label{theorem:main}
If Assumptions \ref{a:regular}, \ref{a:b}, \ref{a:density}, and \ref{a:lyapunov} are satisfied, then 
\begin{equation}\label{eq:asymptotic_normal}
\widehat{Z}_{NT} \stackrel{d}{\rightarrow} N\left(0,  \frac{\xi^2}{\left(m^{-\xi}-1\right)^2}\right).
\end{equation}
\end{theorem}

While the formal proof is deferred to Appendix~\ref{sec:theorem:main}, we briefly highlight its key components here because it is of interest on its own. The statistic $W_{NT}(\tau)$ is decomposed as
$$
W_{NT}(\tau) = A_{1,\tau} + A_{2,\tau} + A_{3,\tau} + A_{4,\tau},
$$
where $A_{1,\tau}$, $A_{2,\tau}$, and $A_{4,\tau}$ are the projections onto $\alpha_i$, $\gamma_t$, and $(\alpha_i, \gamma_t)$, respectively, and $A_{3,\tau}$ denotes the remainder term.
Under two-way clustered sampling, the remainder term $A_{3,\tau}$ is, in general, asymptotically non-Gaussian. However, by considering an appropriate sequence of quantiles with $\tau \to 0$, we can ensure that $A_{3,\tau}$ is asymptotically negligible relative to the leading terms, which are Gaussian.\footnote{Assumption~\ref{a:b} ensures that $\Var(A_{3,\tau}) = o(\tau)$ and $\Var(A_{4,\tau}) = 1$, which together imply that the term $A_{3,\tau}$ is asymptotically negligible relative to the Gaussian component $A_{4,\tau}$.} 
(See the proof of Lemma~\ref{lemma:w} in Appendix~\ref{sec:lemma:w} for further details.)
As a result, we obtain the asymptotic Gaussianity stated in equation~\eqref{eq:asymptotic_normal}, despite the potential complications introduced by two-way clustering.
This robust Gaussianity is a special feature of extreme quantiles, not shared by intermediate quantiles.


The corresponding convergence rate in cross-sectional settings is $\sqrt{\tau N}$ \citep[cf.][]{Chernozhukov2005}. 
By analogy, one might expect a convergence rate of $\sqrt{\min\{\tau N, \tau T\}}$ under two-way clustered sampling. 
However, the actual convergence rate turns out to be faster than this benchmark.
The intuition behind this apparently counterintuitive result is as follows. 
The convergence rate of the estimator $\widehat{Z}_{NT}$ is given by $\displaystyle\frac{a_{NT}}{\sigma_{NT}}$. 
In the non-degenerate case (i.e., when $\Var(W_{NT} \mid \alpha_i) > 0$ or $\Var(W_{NT} \mid \gamma_t) > 0$), we have
$
\sigma_{NT} = O\left( \sqrt{\Var(A_{1,\tau})} \vee \sqrt{\Var(A_{2,\tau})} \right).
$
By Assumption~\ref{a:b}, we obtain $\sqrt{\Var(A_{1,\tau})} = o(T)$ and $\sqrt{\Var(A_{2,\tau})} = o(N)$, which together imply
$
\sigma_{NT} = o\left( \sqrt{N} \vee \sqrt{T} \right).
$
As a result, we obtain
$
\frac{a_{NT}}{\sigma_{NT}} = o\left( \sqrt{\tau N} \vee \sqrt{\tau T} \right),
$
indicating that the rate of $\widehat{Z}_{NT}$ is faster than $\sqrt{\min\{\tau N, \tau T\}}$.
In the degenerate case (i.e., $\Var(W_{NT} \mid \alpha_i) = 0$ and $\Var(W_{NT} \mid \gamma_t) = 0$), we have $\sigma_{NT} = O\left( \sqrt{\Var(A_{4,\tau})} \right) = O(1)$, and hence the convergence rate simplifies to $a_{NT} = \sqrt{\tau NT}$. 
This corresponds to the rate obtained in \citet{Chernozhukov2005} under $NT$ i.i.d.\ observations.

\subsection{Variance Estimation for the Extremal Quantiles}\label{sec:variance_estimation}
As implied by the discussion after Theorem \ref{theorem:main} from the previous subsection, we have the decomposition
\[
\sigma^2_{NT} = \Var\left(W_{NT}(\tau)\right) = \Var(A_{1,\tau} + A_{2,\tau} + A_{3,\tau} + A_{4,\tau}),
\]
where the components are
\begin{align*}
A_{1,\tau} &= \frac{-\sqrt{T}}{\sqrt{\tau N}} \sum_{i=1}^N \mathbb{E}\left[ \tau - \mathbbm{1}\left\{ Y_{it} < \beta(\tau) \right\} \mid \alpha_i \right], \\
A_{2,\tau} &= \frac{-\sqrt{N}}{\sqrt{\tau T}} \sum_{t=1}^T \mathbb{E}\left[ \tau - \mathbbm{1}\left\{ Y_{it} < \beta(\tau) \right\} \mid \gamma_t \right], \\
A_{3,\tau} &= \frac{-1}{\sqrt{\tau NT}} \sum_{i=1}^N \sum_{t=1}^T \Big( \mathbb{E}\left[ \tau - \mathbbm{1}\left\{ Y_{it} < \beta(\tau) \right\} \mid \alpha_i, \gamma_t \right] \\
&\qquad - \mathbb{E}\left[ \tau - \mathbbm{1}\left\{ Y_{it} < \beta(\tau) \right\} \mid \alpha_i \right] 
- \mathbb{E}\left[ \tau - \mathbbm{1}\left\{ Y_{it} < \beta(\tau) \right\} \mid \gamma_t \right] \Big), \qquad\text{and}\\
A_{4,\tau} &= \frac{-1}{\sqrt{\tau NT}} \sum_{i=1}^N \sum_{t=1}^T \left( \left( \tau - \mathbbm{1}\left\{ Y_{it} < \beta(\tau) \right\} \right) - \mathbb{E}\left[ \tau - \mathbbm{1}\left\{ Y_{it} < \beta(\tau) \right\} \mid \alpha_i, \gamma_t \right] \right).
\end{align*}
By construction, the four components are orthogonal, and we have
$$
\Var(A_{1,\tau} + A_{2,\tau} + A_{3,\tau} + A_{4,\tau}) = \Var(A_{1,\tau}) + \Var(A_{2,\tau}) + \Var(A_{3,\tau}) + \Var(A_{4,\tau}).
$$

From the previous subsection, we know that the asymptotic variance of $A_{3,\tau} + A_{4,\tau}$ converges to 1.
On the other hand, the variances of $A_{1,\tau}$ and $A_{2,\tau}$ are given by
\begin{align*}
\Var(A_{1,\tau}) &= \frac{T}{\tau N} \sum_{i=1}^N \mathbb{E} \left[ \left( \mathbb{E}[ D_{it} \mid \alpha_i ] \right)^2 \right] 
= \frac{T}{\tau} \left( -\tau^2 + \mathbb{E} \left[ P\left( Y_{it} < \beta(\tau) \mid \alpha_i \right)^2 \right] \right)
\qquad\text{and}\\
\Var(A_{2,\tau}) &= \frac{N}{\tau T} \sum_{t=1}^T \mathbb{E} \left[ \left( \mathbb{E}[ D_{it} \mid \gamma_t ] \right)^2 \right] 
= \frac{N}{\tau} \left( -\tau^2 + \mathbb{E} \left[ P\left( Y_{it} < \beta(\tau) \mid \gamma_t \right)^2 \right] \right),
\end{align*}
where $D_{it} = \tau - \mathbbm{1}\left\{ Y_{it} < \beta(\tau) \right\}$.
Thus, the asymptotic variance of $W_{NT}(\tau)$ depends on the behavior of $\tau$, $N$, $T$, and the second moments of the conditional probabilities.

For example, if $\tau T = o(1)$ and $\mathbb{E} \left[ P\left(Y_{it} < \beta(\tau) \mid \alpha_i \right)^2 \right] = o(\tau^2)$, then $\Var(A_{1,\tau}) \to 0$. 
Similarly, if $\tau N = o(1)$ and $\mathbb{E} \left[ P\left(Y_{it} < \beta(\tau) \mid \gamma_t \right)^2 \right] = o(\tau^2)$, then $\Var(A_{2,\tau}) \to 0$.
In this case, the entire asymptotic variance of $W_{NT}(\tau)$ is driven by $A_{3,\tau} + A_{4,\tau}$ and converges to 1.
Since the rates of these components are unknown in practice, we propose using the sample counterpart of $\sigma_{NT}$ to construct confidence intervals, rather than relying on its asymptotic approximation.

Note that
$
\left| \mathbb{E}\left[ \tau - \mathbbm{1}\left\{ Y_{it} < \beta(\tau) \right\} \mid \alpha_i \right]
- \mathbb{E}\left[ \tau - \mathbbm{1}\left\{ Y_{it} < \hat{\beta}(\tau) \right\} \mid \alpha_i \right] \right| = o_p(1)
$
uniformly over $i$, due to the smoothness of the conditional CDF and the consistency of $\hat{\beta}(\tau)$.
We are going to claim that 
$
\frac{1}{T} \sum_{t=1}^T \left( \tau - \mathbbm{1}\left\{ Y_{it} < \hat{\beta}(\tau) \right\} \right)
$
is a consistent estimator of $\mathbb{E}\left[ \tau - \mathbbm{1}\left\{ Y_{it} < \hat{\beta}(\tau) \right\} \mid \alpha_i \right]$. To see this, observe that, for each fixed $i$, the observations $(Y_{it})_{t=1}^T$ are conditionally independent given $\alpha_i$. Therefore, for each fixed $i$ and any $y \in \mathbb{R}$,
$
\frac{1}{T} \sum_{t=1}^T \mathbbm{1}\left\{ Y_{it} < y \right\} \stackrel{p}{\to} \mathbb{E}\left[ \mathbbm{1}\left\{ Y_{it} < y \right\} \mid \alpha_i \right]
$
uniformly over $y$, by the uniform law of large numbers for empirical CDFs. Following the same arguments as in Equation (E.1) in the proof of Theorem~5 in \citet{chiang2024standard}, this convergence holds uniformly over $(y, i) \in \mathbb{R} \times \mathbb{N}$.

Analogously, the conditional expectation 
$\mathbb{E} \left[ \tau - \mathbbm{1}\left\{ Y_{it} < \beta(\tau) \right\} \mid \gamma_t \right]$
can be consistently estimated by
$
\frac{1}{N} \sum_{i=1}^N \left( \tau - \mathbbm{1}\left\{ Y_{it} < \hat{\beta}(\tau) \right\} \right).
$

Although the conditional expectation 
$\mathbb{E} \left[ \tau - \mathbbm{1}\left\{ Y_{it} < \beta(\tau) \right\} \mid \alpha_i, \gamma_t \right]$
cannot be directly estimated consistently, it is not necessary to estimate it. Under Assumptions~\ref{a:regular}--\ref{a:lyapunov}, we have $\Var(A_{3,\tau} + A_{4,\tau}) \to 1$.

Putting all these pieces together, we define the variance estimator by
\[
\hat{\sigma}^2_{NT} =
\frac{1}{\tau NT} \sum_{i=1}^N \left( \sum_{t=1}^T \left( \tau - \mathbbm{1}\left\{ Y_{it} < \hat{\beta}(\tau) \right\} \right) \right)^2
+ \frac{1}{\tau NT} \sum_{t=1}^T \left( \sum_{i=1}^N \left( \tau - \mathbbm{1}\left\{ Y_{it} < \hat{\beta}(\tau) \right\} \right) \right)^2
+ 1,
\]
and are going to show that it satisfies the consistency property:
\[
\frac{\hat{\sigma}_{NT}}{\sigma_{NT}} \stackrel{p}{\to} 1.
\]

We now turn to the estimation of the tail index $\xi$ that appears in equation~\eqref{eq:asymptotic_normal}. Define
\[
\begin{aligned} 
\varphi =\frac{(\hat{\beta}(m \tau)-\hat{\beta}(\tau))}{(\beta(m \tau)-\beta(\tau))}
\quad\text{and}\quad
\hat{\rho}_{l} =\frac{(\hat{\beta}(m l \tau)-\hat{\beta}(l \tau))}{(\hat{\beta}(m \tau)-\hat{\beta}(\tau))} .
\end{aligned}\]
Of these, the convergence $\varphi \xrightarrow{p} 1$ follows from the decomposition:
\[
\begin{aligned}
\varphi &= \frac{\hat{\beta}(m \tau) - \hat{\beta}(\tau)}{\beta(m \tau) - \beta(\tau)} \\
&= \frac{\hat{\beta}(m \tau) - \beta(m \tau)}{\beta(m \tau) - \beta(\tau)} 
- \frac{\hat{\beta}(\tau) - \beta(\tau)}{\beta(m \tau) - \beta(\tau)} 
+ \frac{\beta(m \tau) - \beta(\tau)}{\beta(m \tau) - \beta(\tau)} \xrightarrow{p} 1,
\end{aligned}
\]
since the first two terms on the right-hand side are $o_p(\sqrt{\tau N} \vee \sqrt{\tau T})$ by Theorem~1.

Given that $\varphi \xrightarrow{p} 1$, it follows that
\[
- \frac{1}{\log l} \log \hat{\rho}_l \to \xi,
\]
by an argument analogous to the proof of Theorem~6.1, using properties (v) and (vi) from Lemma~9.1 in \citet{Chernozhukov2005}.\footnote{
The only part of Theorem~6.1 in \citet{Chernozhukov2005} where the dependence structure plays a role is in establishing the convergence $\varphi \to 1$. For all other parts, the presence of two-way clustering does not affect the argument. Therefore, once $\varphi \xrightarrow{p} 1$ is established under two-way clustering, it follows that
$
- \frac{1}{\log \ell} \log \hat{\rho}_{\ell} \to \xi.
$
}

Consequently, we have
$
\hat{\xi} \equiv - \frac{\log \hat{\rho}_l}{\log l} \xrightarrow{p} \xi.
$
Also, if we define
$
\hat{a}_{NT} \equiv \frac{\sqrt{\tau NT}}{\hat{\beta}(m \tau) - \hat{\beta}(\tau)},
$
then a consistent estimator of the variance of $\hat{\beta}(\tau)$ is given by
$$
\frac{\hat{\xi}^2}{\left(m^{-\hat{\xi}} - 1\right)^2} \cdot \frac{\hat{\sigma}^2_{NT}}{\hat{a}_{NT}^2}.
$$
We summarize the above result in the following proposition.

\begin{proposition}
If Assumptions \ref{a:regular}, \ref{a:b}, \ref{a:density}, and \ref{a:lyapunov} are satisfied, then
\begin{align*}
  \frac{1}{\var(\hat\beta(\tau))}\cdot \frac{\hat\xi^2}{\left(m^{-\hat\xi}-1\right)^2}\frac{\hat\sigma^2_{NT}}{\hat a^2_{NT}} \stackrel{p}{\to}1.
\end{align*}
\end{proposition}

\subsection{Numerical Examples}
This section presents simulation studies to evaluate the finite-sample performance of the inference procedure based on the asymptotic Gaussianity results. Two designs are employed.

First, consider the following data-generating process:
$$
Y_{it} = \sigma_\alpha \alpha_i + \sigma_\gamma \gamma_t + \sigma_\epsilon \epsilon_{it},
$$
where $\alpha_i \sim \mathcal{N}(0,1)$, $\gamma_t \sim \mathcal{N}(0,1)$, and $\epsilon_{it} \sim \mathcal{N}(0,1)$ are mutually independent. Throughout the simulations, we fix $\sigma_\epsilon = 2.0$, and vary $\sigma_\alpha \in \{1.0, 1.5, 2.0, 2.5, 3.0\}$ and $\sigma_\gamma \in \{1.0, 1.5, 2.0, 2.5, 3.0\}$. The sample size is set to $N = 200$ and $T = 200$, and we consider two extreme quantile levels: $\tau \in \{0.05, 0.01\}$ that are relevant to major empirical applications.

For each Monte Carlo iteration, we compute the point estimate and the corresponding 95\% confidence interval based on the Gaussian approximation as described in Section~\ref{sec:asymptotic_theory} and the variance estimator introduced in Section~\ref{sec:variance_estimation}.

Table~\ref{tab:simulation1} reports the coverage frequencies of the 95\% confidence intervals over 1{,}000 Monte Carlo iterations. 
The left panel reports the results for $\tau = 0.05$, while the right panel reports those for $\tau = 0.01$.
Within each panel, the rows correspond to different values of $\sigma_\alpha \in \{1.0, 1.5, 2.0, 2.5, 3.0\}$, while the columns correspond to different values of $\sigma_\gamma \in \{1.0, 1.5, 2.0, 2.5, 3.0\}$.
The results show that the coverage frequencies are fairly close to the nominal level of 95\%, thereby supporting the validity of the proposed inference method based on asymptotic Gaussianity.

\begin{table}[t]
    \centering
    \caption{Monte Carlo Coverage Frequencies with Nominal Probability of 95\%}
    \label{tab:simulation1}
    \begin{tabular}{lcccccccrlccccc}
        \hline\hline
        & & \multicolumn{5}{c}{$\tau = 0.05$} &&& \multicolumn{5}{c}{$\tau = 0.01$} \\
        \cline{3-7} \cline{10-14}
        & & \multicolumn{5}{c}{$\sigma_\gamma$} &&& \multicolumn{5}{c}{$\sigma_\gamma$} \\
        $\sigma_\alpha$ & & 1.0 & 1.5 & 2.0 & 2.5 & 3.0 &&& 1.0 & 1.5 & 2.0 & 2.5 & 3.0 \\
        \hline
        1.0 & & 0.99 & 0.95 & 0.94 & 0.95 & 0.97 &&& 0.98 & 0.87 & 0.90 & 0.92 & 0.88 \\
        1.5 & & 0.96 & 0.93 & 0.96 & 0.85 & 0.94 &&& 0.91 & 0.90 & 0.91 & 0.94 & 0.90 \\
        2.0 & & 0.87 & 0.94 & 0.96 & 0.95 & 0.92 &&& 0.97 & 0.94 & 0.94 & 0.91 & 0.92 \\
        2.5 & & 0.94 & 0.95 & 0.95 & 0.94 & 0.93 &&& 0.95 & 0.89 & 0.90 & 0.92 & 0.87 \\
        3.0 & & 0.93 & 0.96 & 0.96 & 0.92 & 0.94 &&& 0.96 & 0.94 & 0.90 & 0.93 & 0.93 \\
        \hline\hline
    \end{tabular}
    \caption*{\small\textit{Notes:} The results are shown at quantile levels $\tau=0.05$ (left panel) and $\tau=0.01$ (right panel). Simulation parameters vary over $\sigma_\alpha \in \{1.0,1.5,2.0,2.5,3.0\}$ and $\sigma_\gamma \in \{1.0,1.5,2.0,2.5,3.0\}$. The sample size is $N=T=200$. All results are based on 1{,}000 Monte Carlo replications.}${}$
\end{table}

We next consider an alternative data-generating process:
$$
Y_{it} = \sigma_\alpha \alpha_i \cdot \sigma_\gamma \gamma_t + \sigma_\epsilon \epsilon_{it},
$$
where $\alpha_i \sim \mathcal{N}(0,1)$, $\gamma_t \sim \mathcal{N}(0,1)$, and $\epsilon_{it} \sim \mathcal{N}(0,1)$ are mutually independent. 
In this design, the model incorporates a non-additive interaction between the latent factors $\alpha_i$ and $\gamma_t$. All other aspects of the simulation design remain the same as in the previous setup.

For each Monte Carlo iteration, we compute the point estimate and the corresponding 95\% confidence interval using the asymptotic Gaussian approximation (Section~\ref{sec:asymptotic_theory}) and the variance estimator described in Section~\ref{sec:variance_estimation}.

Table~\ref{tab:simulation2} reports the coverage frequencies of the 95\% confidence intervals across 1{,}000 Monte Carlo iterations. 
The left panel reports the results for $\tau = 0.05$, while the right panel reports those for $\tau = 0.01$.
Within each panel, the rows correspond to different values of $\sigma_\alpha \in \{1.0, 1.5, 2.0, 2.5, 3.0\}$, while the columns correspond to different values of $\sigma_\gamma \in \{1.0, 1.5, 2.0, 2.5, 3.0\}$.
Again, the coverage frequencies are close to the nominal level of 95\%. These results further support the robustness of our inference procedure based on asymptotic Gaussianity.

\begin{table}[t]
    \centering
    \caption{Monte Carlo Coverage Frequencies with Nominal Probability of 95\%}
    \label{tab:simulation2}
    \begin{tabular}{lcccccccrlccccc}
        \hline\hline
        & & \multicolumn{5}{c}{$\tau = 0.05$} &&& \multicolumn{5}{c}{$\tau = 0.01$} \\
        \cline{3-7} \cline{10-14}
        & & \multicolumn{5}{c}{$\sigma_\gamma$} &&& \multicolumn{5}{c}{$\sigma_\gamma$} \\
        $\sigma_\alpha$ & & 1.0 & 1.5 & 2.0 & 2.5 & 3.0 &&& 1.0 & 1.5 & 2.0 & 2.5 & 3.0 \\
        \hline
        1.0 & & 0.99 & 0.95 & 0.94 & 0.95 & 0.97 &&& 0.98 & 0.87 & 0.90 & 0.92 & 0.88 \\
        1.5 & & 0.96 & 0.93 & 0.96 & 0.85 & 0.94 &&& 0.91 & 0.90 & 0.91 & 0.94 & 0.90 \\
        2.0 & & 0.96 & 0.93 & 0.96 & 0.85 & 0.94 &&& 0.97 & 0.94 & 0.94 & 0.91 & 0.92 \\
        2.5 & & 0.94 & 0.95 & 0.95 & 0.94 & 0.93 &&& 0.95 & 0.89 & 0.90 & 0.92 & 0.87 \\
        3.0 & & 0.93 & 0.96 & 0.96 & 0.92 & 0.94 &&& 0.96 & 0.94 & 0.90 & 0.93 & 0.93 \\
        \hline\hline
    \end{tabular}
    \caption*{\small\textit{Notes:} The results are shown at quantile levels $\tau=0.05$ (left panel) and $\tau=0.01$ (right panel). Simulation parameters vary over $\sigma_\alpha \in \{1.0,1.5,2.0,2.5,3.0\}$ and $\sigma_\gamma \in \{1.0,1.5,2.0,2.5,3.0\}$. The sample size is $N=T=200$. All results are based on 1{,}000 Monte Carlo replications.}${}$
\end{table}

\section{Extension to Quantile Regressions}

The baseline analysis thus far focuses on unconditional distributions.
In this section, we extend it to the setting of quantile regression. 
Before proceeding, we reset and redefine the variables and notations.
Let $Y$ denote a response variable taking values in $\mathbb{R}$, and let $X = \left(1, X_{-1}^\prime\right)^\prime$ be a $d \times 1$ vector of regressors. 
Here, $X_{-1}$ denotes the vector $X$ with its first coordinate (corresponding to the constant term) removed. 
Denote the conditional distribution function of $Y$ given $X = x$ by $F_Y(\cdot \mid x)$.

Suppose that we observe data $\{(Y_{it}, X_{it}^\prime) : 1 \le i \le N,\ 1 \le t \le T\}$ generated by
$
(Y_{it}, X_{it}^\prime) = g(\alpha_i, \gamma_t, \varepsilon_{it})
$
for some Borel measurable function $g$, where $\alpha_1, \ldots, \alpha_N$, $\gamma_1, \ldots, \gamma_T$, and $\varepsilon_{11}, \ldots, \varepsilon_{NT}$ are identically distributed and mutually independent random variables. 
As in the unconditional quantile case, this formulates two-way cluster dependence through the latent factors $\alpha_i$ and $\gamma_t$.

Our objective is to conduct inference on features of the conditional quantile function 
\[
F_Y^{-1}(\tau \mid x) = \inf \left\{ y \in \mathbb{R} : F_Y(y \mid x) > \tau \right\},
\]
where $\tau$ is close to zero.
Assume that the conditional quantile function takes the linear form:
\begin{equation}\label{linearmodel}
F_Y^{-1}(\tau \mid x) = X^{\prime} \beta(\tau).
\end{equation}
With this formulation, the coefficient vector $\beta(\tau)$ is the parameter of interest, and it is estimated by
\begin{equation}\label{eq:def_beta_hat_X}
\hat{\beta}(\tau) \in \arg \min_{\beta \in \mathbb{R}^d} \sum_{i=1}^N \sum_{t=1}^T \rho_\tau\left(Y_{it} - X_{it}^\prime \beta\right),
\end{equation}
where $\rho_tau$ defined by $\rho_\tau(u) = u(\tau - \mathbbm{1}\{u < 0\})$ denotes the check function.

\subsection{Asymptotic Theory for the Quantile Regressions}

As in the unconditional quantile case, we are interested in the limiting distribution of the statistic
\[
\widehat{Z}_{NT} = r_{NT} \left( \hat{\beta}(\tau) - \beta(\tau) \right),
\]
where $\{r_{NT}\}$ is a sequence that governs the rate of convergence. Our focus is on the asymptotic behavior of $\widehat{Z}_{NT}$, and thus of $\hat{\beta}(\tau)$, in the left tail of the conditional distribution of $Y$ given $X$, in the limit as $\tau \to 0$.

For a random variable $u$ with its distribution function $F_u$, let the lower bound of its support be denoted by $s_u = 0$ or $-\infty$, for notational convenience and without loss of generality by location normalization.
We reintroduce the domain-of-attraction condition in three types:
\begin{enumerate}[${}$\quad Type 1: \ ]
\item
as $z \searrow s_u=0$ or $-\infty$,
$
F_u(z+v a(z)) \sim F_u(z) e^v
$;
\item
as $z \searrow s_u=-\infty$, \hspace{7mm}
$
F_u(v z) \sim v^{-1 / \xi} F_u(z)
$; \qquad\text{and}
\item 
as $z \searrow s_u=0$, \hspace{12mm}
$
F_u(v z) \sim v^{-1 / \xi} F_u(z).
$
\end{enumerate}
where $a(z) \equiv \int_{s_u}^y F_u(z) d v / F_u(z) \text { for } z>s_u$
With this definition in place, we now state the three assumptions inherited from \citet{Chernozhukov2005}.

\begin{assumption}\label{line}
There exists an auxiliary line $x \mapsto x^{\prime} \beta_r$ such that
$
U \equiv Y - X^{\prime} \beta_r,
$
with lower bound $s_u = 0$ or $s_u = -\infty$, and such that the distribution function $F_u$ of $U$ satisfies a domain-of-attraction condition of Type 1, 2, or 3.
Moreover, the conditional distribution of $U$ given $X = x$ satisfies
\[
F_U(z \mid x) \sim K(x) \cdot F_u(z) \quad \text{uniformly in } x \in \mathbf{X} \text{ as } z \searrow s_U
\]
for some continuous and bounded function $K(\cdot) > 0$. Without loss of generality, we normalize $K(x) = 1$ at $x = \mu_X$, and define $F_u(z) \equiv F_U(z \mid \mu_X)$.
\end{assumption}

\begin{assumption}\label{x}  
The distribution function $F_X$ of $X = \left(1, X_{-1}^\prime\right)^\prime$ has a compact support $\mathbf{X} \subset \mathbb{R}^d$, and ${E}[X X^\prime]$ is positive definite. Without loss of generality, we normalize the mean vector to $\mu_X = {E}[X] = (1, 0, \ldots, 0)^\prime$.
\end{assumption}

\begin{assumption}\label{regular2}
(i) $\displaystyle\frac{\partial F_U^{-1}(\tau \mid x)}{\partial \tau} \sim \displaystyle\frac{\partial F_u^{-1}(\tau / K(x))}{\partial \tau} \quad$ uniformly for all $x \in \mathbf{X}$, and
(ii) $\displaystyle\frac{\partial F_u^{-1}(\tau)}{\partial \tau}$
is regularly varying at 0 with exponent $-\xi-1$.
\end{assumption}

Assumptions~\ref{line}--\ref{regular2} correspond to Conditions R1--R3 in \citet{Chernozhukov2005}, respectively. Under these assumptions, Theorem~3.1 of \citet{Chernozhukov2005} implies that
\[
K(x) =
\begin{cases}
e^{-x^{\prime} \mathbf{c}} & \text{if } F_u \text{ is of Type 1 } (\xi = 0), \\
\left( x^{\prime} \mathbf{c} \right)^{1/\xi} & \text{if } F_u \text{ is of Type 2 } (\xi > 0), \\
\left( x^{\prime} \mathbf{c} \right)^{1/\xi} & \text{if } F_u \text{ is of Type 3 } (\xi < 0),
\end{cases}
\]
for some $\mathbf{c} \in \mathbb{R}^d$.

With the above tools, we now turn to the asymptotic theory in our framework.
Define
$$
W_{NT}(\tau) \equiv \frac{-1}{\sqrt{\tau NT}} \sum_{i=1}^N \sum_{t=1}^T \left( \tau - \mathbbm{1}\left\{ Y_{it} < X_{it}^\prime \beta(\tau) \right\} \right) X_{it},
$$
and denote its variance by
$$
\Sigma_{NT} = \Var \left( W_{NT}(\tau) \right).
$$
Also, define the scaling factor
$$
a_{NT} \equiv \frac{\sqrt{\tau NT}}{\mu_X^\prime \left( \beta(m \tau) - \beta(\tau) \right)}, \qquad \mu_X^\prime = \mathbb{E}[X^\prime],
$$
and the normalized statistic
$$
\widehat{Z}_{NT} \equiv a_{NT} \cdot \Sigma_{NT}^{-1/2} \cdot \mathcal{Q}_H \cdot \left( \hat{\beta}(\tau) - \beta(\tau) \right),
$$
where
$$
\mathcal{Q}_H \equiv \mathbb{E} \left[ H(X)^{-1} X X^\prime \right]
\quad \text{with} \quad
H(x) \equiv 
\begin{cases}
x^\prime \mathbf{c}, & \text{for Type 2 and Type 3 tails}, \\
1, & \text{for Type 1 tails}.
\end{cases}
$$

To establish the asymptotic Gaussianity of $\widehat{Z}_{NT}$, we first consider the convergence of an intermediate statistic:
$$
\tilde{Z}_{NT} \equiv \frac{a_{NT}}{c_{NT}} \left( \hat{\beta}(\tau) - \beta(\tau) \right),
$$
where $c_{NT} \in \mathbb{R}_+$ is a scaling factor satisfying
$$
\frac{c_{NT}}{a_{NT}} = o(1) \qquad \text{and} \qquad \frac{\Sigma_{NT}}{c_{NT}^2} = O(1).
$$
Assumption~\ref{b2}, introduced later, ensures the existence of such a scalar $c_{NT}$, as it guarantees that $\Sigma_{NT}$ converges faster than $\tau NT$. 
We will show that $\tilde{Z}_{NT}$ is asymptotically Gaussian, and it then follows that $\widehat{Z}_{NT}$ also inherits asymptotic Gaussianity.

Similarly to the unconditional quantile case, it follows from equation~\eqref{eq:def_beta_hat_X} that the transformation $\tilde{Z}_{NT}$ of $\hat{\beta}(\tau)$ minimizes the objective function
\begin{equation}\label{eq:q_original_X}
Q_{NT}(z, \tau) \equiv \frac{a_{NT}}{c^2_{NT} \sqrt{\tau NT}} \sum_{i=1}^N \sum_{t=1}^T \left( \rho_\tau\left( Y_{it} - X_{it}^\prime \beta(\tau) - \frac{c_{NT} X_{it}^\prime z}{a_{NT}} \right) - \rho_\tau\left( Y_{it} - X_{it}^\prime \beta(\tau) \right) \right).
\end{equation}
Using Knight's identity, the objective \eqref{eq:q_original_X} can be rewritten as
\begin{equation}\label{eq:q_decomposition2}
Q_{NT}(z, \tau) = W_{NT}(\tau)^\prime z + G_{NT}(z, \tau),
\end{equation}
where
\[
W_{NT}(\tau) \equiv \frac{-1}{\sqrt{\tau NT}} \sum_{i=1}^N \sum_{t=1}^T \left( \tau - \mathbbm{1}\left\{ Y_{it} < X_{it}^\prime \beta(\tau) \right\} \right) X_{it}
\]
and
\[
G_{NT}(z, \tau) \equiv \frac{a_{NT}}{\sqrt{\tau NT}} \sum_{i=1}^N \sum_{t=1}^T \int_0^{X_{it}^\prime z / a_{NT}} \left( \mathbbm{1}\left\{ Y_{it} - X_{it}^\prime \beta(\tau) \le s \right\} - \mathbbm{1}\left\{ Y_{it} - X_{it}^\prime \beta(\tau) \le 0 \right\} \right) ds.
\]

We have already stated Assumptions~\ref{line}--\ref{regular2}, which are inherited from \citet[][Conditions R1--R3]{Chernozhukov2005}. In addition, we introduce a new assumption concerning the tail behavior of the conditional distributions given $\alpha_i$, $\gamma_t$, and both $\alpha_i$ and $\gamma_t$, which is specific to the two-way clustered data structure considered in our setting.

\begin{assumption}\label{conditional tail}
$F_u(z|\alpha_i)$, $F_u(z|\gamma_t)$, and $F_u(z|\alpha_i,\gamma_t)$ are of Type 1, 2, or 3. Furthermore, 
\begin{align*}
& F_U(z \mid x,\alpha_i) \sim K_\alpha(x) \cdot F_u(z|\alpha_i) \quad \text{uniformly for all }  x \in \mathbf{X} \text{ as }  z \searrow s_U;
\\
& F_U(z \mid x,\gamma_t) \sim K_\gamma(x) \cdot F_u(z|\gamma_t) \quad \text{uniformly for all }  x \in \mathbf{X} \text{ as }  z \searrow s_U; \text{ and }
\\
& F_U(z \mid x,\alpha_i,\gamma_t) \sim K_{\alpha,\gamma}(x) \cdot F_u(z|\alpha_i,\gamma_t) \quad \text{uniformly for all }  x \in \mathbf{X} \text{ as }  z \searrow s_U.
\end{align*}
\end{assumption}

Assumption~\ref{conditional tail} requires that Assumption~\ref{line} continues to hold when conditioning on each of the variables $\alpha_i$, $\gamma_t$, and both $\alpha_i$ and $\gamma_t$, thereby accommodating the dependence structure inherent in two-way clustered data.
The corresponding scaling functions $K_\alpha$, $K_\gamma$, and $K_{\alpha,\gamma}$ are allowed to differ from one another.

Finally, we introduce the following set of assumptions, which are analogous to those stated for the analysis of unconditional quantiles in Section \ref{sec:asymptotic_theory}.

\begin{assumption}\label{conditional regular}
As $\tau\to 0$ the following conditions hold:\\
(i) 
$\displaystyle\frac{\partial F_U^{-1}(\tau \mid x,\alpha_i)}{\partial \tau} \sim \displaystyle\frac{\partial F_u^{-1}(\tau / K_\alpha(x)\mid \alpha_i)}{\partial \tau} \quad$ uniformly for all $x \in \mathbf{X}$;\\${}$\hspace{0.45cm}
$\displaystyle\frac{\partial F_U^{-1}(\tau \mid x,\gamma_t)}{\partial \tau} \sim \displaystyle\frac{\partial F_u^{-1}(\tau / K_\gamma(x)\mid \gamma_t)}{\partial \tau} \quad$ uniformly for all $x \in \mathbf{X}$; \text{ and }\\${}$\hspace{0.45cm}
$\displaystyle\frac{\partial F_U^{-1}(\tau \mid x,\alpha_i,\gamma_t)}{\partial \tau} \sim \displaystyle\frac{\partial F_u^{-1}(\tau / K_{\alpha,\gamma}(x)\mid \alpha_i,\gamma_t)}{\partial \tau} \quad$ uniformly for all $x \in \mathbf{X}$.\\
(ii) $\displaystyle\frac{\partial F_u^{-1}(\tau\mid \alpha_i)}{\partial \tau}$,
 $\displaystyle\frac{\partial F_u^{-1}(\tau\mid \gamma_t)}{\partial \tau}$, and 
  $\displaystyle\frac{\partial F_u^{-1}(\tau\mid \alpha_i,\gamma_t)}{\partial \tau}$
are all regularly varying at 0 .
\end{assumption} 

\begin{assumption}\label{b2}
Let $D_{it}(\tau)=(\tau-\mathbbm{1}\left[Y_.<X_{it}^{\prime}\beta(\tau)\right])X_{it}$. As $\tau\to 0$, $N,T \to \infty$, and $\tau N T \to \infty$, the following conditions hold:
(i)$E[E[D_{it}(\tau)|\alpha_i]\cdot E[D^\prime_{it}(\tau)|\alpha_i]]=O(\tau\land \tau^2 N)$; (ii)$E[E[D_{it}(\tau)|\gamma_t]\cdot E[D^\prime_{it}(\tau)|\gamma_t]]=O(\tau\land \tau^2 N)$; and (iii)$E[E[D_{it}(\tau)|\alpha_i,\gamma_t]\cdot E[D^\prime_{it}(\tau)|\alpha_i,\gamma_t]]=o(\tau)$.
\end{assumption}

\begin{assumption}\label{density2}
As $u\to s_u$ the following conditions hold:\\
(i) $\displaystyle\frac{f_u(u|\alpha_i)}{f_u(u)}=O(1)$; 
(ii) $\displaystyle\frac{f_u(u|\gamma_t)}{f_u(u)}=O(1)$; and 
(iii) $\displaystyle\frac{f_u(u|\alpha_i,\gamma_t)}{f_u(u)}=O(1)$.
\end{assumption}

\begin{assumption}\label{lyapunov2}
As $\tau\to 0$ and $N,T \to \infty$, the following conditions hold:\\
(i) $\displaystyle\frac{E[E\norm{\left(\mathbbm{1}\left[Y_{it}< X_{it}^{\prime}\beta(\tau)\right]\right)X_{it} |\alpha_i}
^3]}{E[E\norm {\left(\mathbbm{1}\left[Y_{it}< X_{it}^{\prime}\beta(\tau)\right]\right) X_{it}|\alpha_i}^2]^{\frac{3}{2}}}=o(\sqrt{N})$;\\
(ii)
 $\displaystyle\frac{E[E\norm{\left(\mathbbm{1}\left[Y_{it}< X_{it}^{\prime}\beta(\tau)\right]\right)X_{it} |\gamma_t}^3]}{E[E\norm{\left(\mathbbm{1}\left[Y_{it}< X_{it}^{\prime}\beta(\tau)\right]\right) X_{it}|\gamma_t}^2]^{\frac{3}{2}}}=o(\sqrt{T})$; and \\
(iii)$\displaystyle\frac{ E[E\norm{\left(\mathbbm{1}\left[Y_{it}< X_{it}^{\prime}\beta(\tau)\right]\right)X_{it} |\alpha_i, \gamma_t}^3 |\alpha_i,\gamma_t]}{  E[E\norm{\left(\mathbbm{1}\left[Y_{it}< X_{it}^{\prime}\beta(\tau)\right]\right) X_{it}|\alpha_i, \gamma_t}^2 |\alpha_i,\gamma_t]^{\frac{3}{2}}}=o(\sqrt{NT}).$
\end{assumption}

\bigskip
Assumptions~\ref{conditional regular}--\ref{lyapunov2} are analogous to Assumptions~\ref{a:regular}--\ref{a:lyapunov}, respectively, in Section \ref{sec:asymptotic_theory} concerning unconditional quantiles.
Assumption~\ref{conditional regular} (i) requires that Assumption~\ref{regular2} holds conditional on each of the latent variables $\alpha_i$, $\gamma_t$, and both $\alpha_i$ and $\gamma_t$.
Assumption~\ref{conditional regular} (ii) requires the existence of the quantile density function and its regular variation. It also imposes regular variation on the conditional quantile density functions. The exponents of the regular variation may differ across the conditional densities $f_u(F_u^{-1}(\tau) \mid \alpha_i)$, $f_u(F_u^{-1}(\tau) \mid \gamma_t)$, and $f_u(F_u^{-1}(\tau) \mid \alpha_i, \gamma_t)$.
Similar to Assumption~\ref{a:b}, Assumption~\ref{b2} here imposes constraints on the conditional cumulative distribution function of the variable $u$ given $X$. 
Assumption~\ref{density2} governs the relationship between the unconditional and conditional density functions of $u$. For example, Assumption~\ref{density2} (i) requires that the density of $u$ does not decay significantly faster than its conditional counterpart given $\alpha_i$. In other words, the decline in the density of $u$ should not be much more rapid than the decline in the conditional density of $u$ given $\alpha_i$.
Lastly, Assumption~\ref{lyapunov2} imposes certain moment conditions needed to verify the Lyapunov-type central limit theorem, which underpins the asymptotic Gaussianity of our estimator.

Under Assumptions~\ref{conditional regular}--\ref{lyapunov2}, and analogously to the unconditional quantile case, we can show that $c_{NT}^{-1} W_{NT}(\tau)$ converges in distribution to a multivariate Gaussian distribution as
$$
c_{NT}^{-1} W_{NT}(\tau) \xrightarrow{d} \mathcal{N}(0, \Gamma),
$$
where
$$
\Gamma = \lim_{N, T \to \infty,\ \tau \to 0} \frac{\Sigma_{NT}}{c_{NT}^2}.
$$
In addition, we can establish the convergence of the second component:
$$
G_{NT}(z, \tau) \xrightarrow{p} \mathbb{E}[G_{NT}(z, \tau)] = \frac{1}{2} \cdot \frac{m^{-\xi} - 1}{-\xi} \cdot z^{\prime} \mathcal{Q}_H z.
$$
See Lemmas~\ref{W2} and~\ref{G2} in Appendices~\ref{sec:lemma:regression:w} and~\ref{sec:lemma:regression:g}, respectively, for formal and detailed derivations of these results.

Consequently, we obtain the following limiting distribution for the solution $\tilde{Z}_{NT}$ to the minimization problem $\min_{z \in \mathbb{R}^d} Q_{NT}(z, \tau)$, where $Q_{NT}(z, \tau) = c_{NT}^{-1} W_{NT}(\tau)^\prime z + G_{NT}(z, \tau)$, as defined in equation~\eqref{eq:q_decomposition2}.
Since the limiting distribution of $\tilde{Z}_{NT}$ is Gaussian, it follows immediately that $\widehat{Z}_{NT}$ also converges in distribution to a Gaussian distribution.

\begin{theorem}\label{theorem:regression}
If Assumptions \ref{line}-\ref{lyapunov2} are satisfied, then
\begin{equation*}
\widehat{Z}_{NT} \stackrel{d}{\rightarrow} N\left(0, \Omega_0\right),
\end{equation*}
where
\begin{align*}
\Omega_0 \equiv& \mathcal{Q}_H^{-1} \Sigma \mathcal{Q}_H^{-1} \frac{\xi^2}{\left(m^{-\xi}-1\right)^2}
\qquad\text{and}\qquad
\mathcal{Q}_H \equiv E[[H(X)]^{-1} X X^{\prime}]
\end{align*}
with 
$H(x) \equiv 1$ under Type 1 and
$H(x) \equiv x^{\prime} \mathbf{c}$
under Types 2--3.
\end{theorem}

See Appendix~\ref{sec:theorem:regression} for a detailed proof. Despite some additional complications, the key features of the argument remain essentially the same as those presented in the unconditional quantile case in Section   \ref{sec:asymptotic_theory}.

\subsection{Variance Estimation for the Quantile Regressions}
Similarly to the unconditional quantile case presented in Section \ref{sec:variance_estimation}, we use the decomposition
\[
\Sigma_{NT} = \Var\left(W_{NT}(\tau)\right) = \Var(A_{1,\tau} + A_{2,\tau} + A_{3,\tau} + A_{4,\tau}),
\]
where
\begin{equation*}
\begin{aligned}
A_{1,\tau} =&\frac{-\sqrt{T}}{\sqrt{\tau N}}\sum_{i=1}^N E[\left(\tau-\mathbbm{1}\left[Y_{it}<X_{it}^{\prime} \beta(\tau)\right]\right) X_{it}|\alpha_i],\\
A_{2,\tau} =&\frac{-\sqrt{N}}{\sqrt{\tau T}}\sum_{t=1}^T E[\left(\tau-\mathbbm{1}\left[Y_{it}<X_{it}^{\prime} \beta(\tau)\right]\right) X_{it}|\gamma_{t}],\\
A_{3,\tau} =&\frac{-1}{\sqrt{\tau N T}}\sum_{i=1}^N \sum_{t=1}^T (E[\left(\tau-\mathbbm{1}\left[Y_{it}<X_{it}^{\prime} \beta(\tau)\right]\right) X_{it}|\alpha_i, \gamma_t]-E[\left(\tau-\mathbbm{1}\left[Y_{it}<X_{it}^{\prime} \beta(\tau)\right]\right) X_{it}|\alpha_i]\\
&\:\:-E[\left(\tau-\mathbbm{1}\left[Y_{it}<X_{it}^{\prime} \beta(\tau)\right]\right) X_{it}|\gamma_t]), \qquad\text{and}\\
A_{4,\tau} =&\frac{-1}{\sqrt{\tau NT}}\sum_{i=1}^N \sum_{t=1}^T(\left(\tau-\mathbbm{1}\left[Y_{it}<X_{it}^{\prime} \beta(\tau)\right]\right)X_{it}-E[\left(\tau-\mathbbm{1}\left[Y_{it}<X_{it}^{\prime} \beta(\tau)\right]\right) X_{it}|\alpha_i, \gamma_t]).
\end{aligned}
\end{equation*}
By construction, the four components are orthogonal, and we have
$$
\Var(A_{1,\tau} + A_{2,\tau} + A_{3,\tau} + A_{4,\tau}) = \Var(A_{1,\tau}) + \Var(A_{2,\tau}) + \Var(A_{3,\tau}) + \Var(A_{4,\tau}).
$$
From the previous subsubsection, we know that the asymptotic variance of $A_{3,\tau} + A_{4,\tau}$ is $E[X_{it} X^\prime_{it}]$.
On the other hand, the variances of $A_{1,\tau}$ and $A_{2,\tau}$ are given by
\begin{align*}
\Var(A_{1,\tau}) &=\frac{T}{\tau N} \sum_{i=1}^N E[E[D_{it}|\alpha_i] E[D^\prime_{it}|\alpha_i]] \qquad\text{and} \\
\Var(A_{2,\tau}) &= \frac{N}{\tau T} \sum_{t=1}^T E[E[D_{it}|\gamma_t] E[D^\prime_{it}|\gamma_t]],
\end{align*}
where $D_{it} = \tau - \mathbbm{1}\left\{ Y_{it} < X^\prime_{it}\beta(\tau) \right\}$.
As in the case of unconditional quantiles,  the asymptotic variance of $W_{NT}$ depends on the behavior of $\tau$, $N$, and $T$.
Since the rates of these components
are unknown in practice, we propose using the sample counterpart of $\Sigma_{NT}$ to construct confidence intervals, rather than relying on its asymptotic approximation.

Since $E[E[D_{it}|\alpha_i] E[D^\prime_{it}|\alpha_i]]$ and $E[E[D_{it}|\gamma_t] E[D^\prime_{it}|\gamma_t]]$  can be estimated with \\
\begin{align*}
&\displaystyle\frac{1}{N}\sum_{i=1}^{N}\left[\left(\frac{1}{{ T}}\sum_{t=1}^T \left(\tau-\mathbbm{1}\left[Y_{it}< X^\prime_{it} \hat \beta(\tau)\right]\right)X^\prime_{it}\right)\left(\frac{1}{{ T}}\sum_{t=1}^T \left(\tau-\mathbbm{1}\left[Y_{it}< X^\prime_{it} \hat \beta(\tau)\right]\right)X^\prime_{it}\right)^\prime \right]
\text{ and} 
\\
&\displaystyle\frac{1}{T}\sum_{t=1}^{T}\left[\left(\frac{1}{{ N}}\sum_{i=1}^N\left(\tau-\mathbbm{1}\left[Y_{it}< X^\prime_{it} \hat \beta(\tau)\right]\right)X^\prime_{it}\right)\left(\frac{1}{{ N}}\sum_{i=1}^N\left(\tau-\mathbbm{1}\left[Y_{it}< X^\prime_{it} \hat \beta(\tau)\right]\right)X^\prime_{it}\right)^\prime\right],
\end{align*}
respectivly, the variances of $A_{1,\tau}$ and $A_{2,\tau}$ can be estimated by
\begin{align*}
\widehat{\Var(A_{1,\tau})}=&\frac{1}{\tau NT}\sum_{i=1}^{N}\left[\left(\sum_{t=1}^T \left(\tau-\mathbbm{1}\left[Y_{it}< X^\prime_{it} \hat \beta(\tau)\right]\right)X^\prime_{it}\right)\left(\sum_{t=1}^T \left(\tau-\mathbbm{1}\left[Y_{it}< X^\prime_{it} \hat \beta(\tau)\right]\right)X^\prime_{it}\right)^\prime \right]
\text{ and }\\
\widehat{\Var(A_{2,\tau})}=&\frac{1}{\tau NT}\sum_{t=1}^{T}\left[\left(\sum_{i=1}^N\left(\tau-\mathbbm{1}\left[Y_{it}< X^\prime_{it} \hat \beta(\tau)\right]\right)X^\prime_{it}\right)\left(\sum_{i=1}^N\left(\tau-\mathbbm{1}\left[Y_{it}< X^\prime_{it} \hat \beta(\tau)\right]\right)X^\prime_{it}\right)^\prime\right],
\end{align*}
respectively. 
Although $E[E[D_{it}|\alpha_i, \gamma_t] E[D^\prime_{it}|\alpha_i, \gamma_t]]$ cannot be directly estimated consistently, it is not necessary to estimate.
Therefore, we can construct $\hat\Sigma_{NT}$ as 
\begin{align*}
\hat\Sigma_{NT}=&\frac{1}{\tau NT}\sum_{i=1}^{N}\left[\left(\sum_{t=1}^T \left(\tau-\mathbbm{1}\left[Y_{it}< X^\prime_{it} \hat \beta(\tau)\right]\right)X^\prime_{it}\right)\left(\sum_{t=1}^T \left(\tau-\mathbbm{1}\left[Y_{it}< X^\prime_{it} \hat \beta(\tau)\right]\right)X^\prime_{it}\right)^\prime \right]\\
&+
\frac{1}{\tau NT}\sum_{t=1}^{T}\left[\left(\sum_{i=1}^N\left(\tau-\mathbbm{1}\left[Y_{it}< X^\prime_{it} \hat \beta(\tau)\right]\right)X^\prime_{it}\right)\left(\sum_{i=1}^N\left(\tau-\mathbbm{1}\left[Y_{it}< \hat 
\beta(\tau)\right]\right)X^\prime_{it}\right)^\prime\right]\\
&+ \frac{1}{NT}\sum_{i=1}^N \sum_{t=1}^T X_{it} X^\prime_{it},
\end{align*}
and show its consistency property 
\[
\Sigma_{NT}^{-1} \hat\Sigma_{NT} \xrightarrow{p} 1.
\]

To estimate the tail index $\xi$, for any $x,\dot{x} \in \mathcal{X}$, define
\[
\varphi \equiv \frac{x^{\prime} \left( \hat{\beta}(m \tau) - \hat{\beta}(\tau) \right)}{x^{\prime} \left( \beta(m \tau) - \beta(\tau) \right)}
\qquad\text{and}\qquad
\hat{\rho}_{x, \dot{x}, l} \equiv \frac{x^{\prime} \left( \hat{\beta}(m l \tau) - \hat{\beta}(l \tau) \right)}{\dot{x}^{\prime} \left( \hat{\beta}(m \tau) - \hat{\beta}(\tau) \right)}.
\]
As in the case of unconditional quantiles, we have $\varphi \xrightarrow{p} 1$ by Theorem~2. Consequently,
\[
- \frac{1}{\log l} \log \hat{\rho}_{\bar{X}, \bar{X}, l} \to \xi
\]
follows by applying the same argument as in the unconditional quantile case.
Also, define 
\[
\hat a_{NT} = \frac{\sqrt{\tau NT}}{(\hat\beta(m \tau)-\hat\beta(\tau))}
\]
and
\[
\widehat{\mathcal{Q}}_H = \frac{1}{NT}\sum_{i=1}^N\sum_{t=1}^T \left(\hat{\rho}^{-1}_{X_{it}, \bar X, 1} X_{it} X^\prime_{it}\right).
\]
Thus, we obtain the following result.
\begin{proposition}If Assumptions \ref{line}-\ref{lyapunov2} are satisfied, then
    \[
    Var(\hat\beta)^{-1}\cdot \hat a_{NT} \cdot \widehat{\mathcal{Q}}_H^{-1}\cdot \widehat\Sigma^{1/2}_{NT}\cdot\widehat{\mathcal{Q}}_H^{-1}\cdot \frac{\hat\xi^2}{\left(m^{-\hat\xi}-1\right)^2 }
    \xrightarrow{p} 1.
    \]
\end{proposition}

\section{Empirical Application: Growth-at-Risk}

In this section, we demonstrate an application of our proposed method to explore the growth-at-risk (GaR) framework of \citet*{adrian2022term}, and to analyze the role of financial conditions in shaping the distribution of forecast GDP growth across 11 advanced economies (AEs). These 11 AEs are identified by the IMF as having systemically important financial sectors.

Following \citet{adrian2022term}, we estimate the dynamics of the GDP growth distribution over horizons ranging from 1 to 12 quarters using local projections in a panel setting. Specifically, we examine the conditional distributions of GDP growth for both near-term (1--4 quarters ahead) and medium-term (5--12 quarters ahead) forecast horizons.
In conducting the inference, we account for the two-way clustering structure inherent in the panel data, where two ways of strong dependence arise from country-specific factors $\alpha_i$ and time-specific factors $\gamma_t$.

Let $\Delta y_{i, t+h}$ denote the annualized average GDP growth rate for country $i$ from quarter $t$ to quarter $t+h$. Let $x_{i, t}$ denote a vector of conditioning variables, which includes current GDP growth, the inflation rate, and the financial conditions index \citep[FCI; cf.][]{adrian2022term}.

Our estimation and inference procedure is based on Equation (1) in \citet{adrian2022term}. Specifically, we estimate the conditional quantile of $\Delta y_{i, t+h}$ via panel quantile regression on $x_{i, t}$. The quantile-specific regression slope $\delta_\tau^{(h)}$ is estimated by
\[
\hat{\delta}_\tau^{(h)} = \argmin_{\delta \in \mathbb{R}^d} \sum_{i=1}^{N} \sum_{t=1}^{T-h} \rho_\tau \left( \Delta y_{i, t+h} - x_{i, t}^\prime \delta \right),
\]
where $\rho_\tau$ defined by $\rho_\tau(u) = u(\tau - \mathbbm{1}\{u < 0\})$ is the check function.
For inference, we apply our Gaussian limit distribution together with our proposed two-way cluster-robust standard error estimator specifically developed for extremal quantiles.

Table \ref{tab:gdp_growth} presents the estimates of the coefficients for various horizons ($h=2$, $h=4$, $h=6$, $h=8$, and $h=10$) focusing on the 5\% conditional quantile, together with their two-way cluster-robust standard errors. The table includes the coefficients for the FCI, current GDP growth, and inflation across the horizons $h \in \{2,4,8,10,12\}$ and quantiles $\tau \in \{0.04,0.02,0.01\}$.

\begin{table}
    \centering
    \begin{tabular}{lcccccc}
        \hline\hline
        & $h=2$ & $h=4$ & $h=6$ & $h=8$ & $h=10$ & $h=12$ \\
        \hline
        \multicolumn{7}{l}{$\tau = 0.01$} \\
        \text{GDP growth} & 0.261 & 0.101 & -0.024 & 0.015 & 0.069 & 0.053 \\
        & (1.097) & (1.022) & (0.247) & (0.456) & (0.841) & (0.415) \\
        \text{Inflation} & -0.178 & 0.039 & 0.106 & 0.019 & -0.057 & -0.029 \\
        & (0.948) & (0.378) & (0.674) & (0.570) & (0.657) & (0.556) \\
        \text{FCI} & -1.640 & -1.338 & -0.386 & 0.341 & 0.395 & 0.157 \\
        & (3.299) & (2.537) & (1.412) & (1.945) & (2.498) & (2.506) \\
        \hline
        \multicolumn{7}{l}{$\tau = 0.02$} \\
        \text{GDP growth} & 0.235 & 0.303 & 0.115 & 0.101 & 0.115 & 0.113 \\
        & (0.362) & (0.271) & (0.149) & (0.672) & (0.228) & (0.159) \\
        \text{Inflation} & -0.182 & -0.124 & -0.015 & 0.037 & -0.030 & 0.021 \\
        & (0.340) & (0.324) & (0.135) & (0.374) & (0.228) & (0.165) \\
        \text{FCI} & -1.256 & -1.040 & -0.337 & 0.201 & 0.350 & 0.196 \\
        & (1.123) & (1.020) & (0.896) & (0.737) & (0.883) & (0.931) \\
        \hline
        \multicolumn{7}{l}{$\tau = 0.04$} \\
        \text{GDP growth} & 0.247 & 0.252 & 0.220 & 0.146 & 0.146 & 0.148 \\
        & (0.071) & (0.048) & (0.046) & (0.032) & (0.043) & (0.027) \\
        \text{Inflation} & -0.189 & -0.159 & -0.104 & -0.048 & -0.006 & 0.034 \\
        & (0.075) & (0.070) & (0.054) & (0.030) & (0.037) & (0.024) \\
        \text{FCI} & -1.225 & -0.884 & -0.251 & 0.272 & 0.349 & 0.301 \\
        & (0.239) & (0.189) & (0.208) & (0.040) & (0.184) & (0.180) \\
        \hline\hline
    \end{tabular}
    \caption{Estimates and two-way cluster-robust standard errors for extreme quantiles. The coefficients represent the effects of the FCI, current GDP growth, and inflation on GDP growth in near- and medium-term horizons.}\label{tab:gdp_growth}
\end{table}


Observe that the FCI has a significantly negative impact on GDP growth at near-term horizons ($h = 2$ and $h = 4$), while its effects turn positive at medium-term horizons ($h = 8$ and $h = 10$). This pattern suggests that looser financial conditions are associated with higher short-run growth but may be linked to increased downside risk over the medium term.
We emphasize that these significant results are obtained even after accounting for the two-way clustering structure in the panel data as well as the extreme quantiles.

The coefficients on current GDP growth (\texttt{dlgdp}) are positive and statistically significant across all forecast horizons, indicating that recent economic momentum is a strong predictor of future growth. Inflation (\texttt{infl}) generally exerts a negative effect on GDP growth, although the magnitude of this impact tends to diminish at longer horizons.

These empirical results are broadly consistent with \citet{adrian2022term}, and our robust inference results reconfirm the earlier findings. The value added by our analysis is thus the guaranteed robustness of these empirical conclusions with statistical inference accounting for the two-way clustering and extreme quantiles.

\section{Summary}

This study investigates extreme quantiles in the context of data sampled under two-way clustering. Our objective is to characterize the asymptotic properties of extreme quantiles in the presence of such dependence structures.

We begin by analyzing the behavior of unconditional quantiles in the tails when the data exhibit two-way clustering, and subsequently extend our analysis to the case of quantile regression. As noted by \citet{Menzel2021}, the presence of two-way clustering can lead to asymptotic distributions of sample means and related statistics that deviate from Gaussianity. In this study, we demonstrate that extreme quantiles retain asymptotic Gaussianity even under two-way clustering for extremal quantiles of intermediate orders. This result enables valid and robust inference for tail behavior despite the complex dependence structure.
Simulation studies corroborate our theoretical findings.

Accounting for both two-way clustering and tail estimation is essential in applied research. By incorporating these considerations, our framework provides a more robust and reliable basis for inference on the extremal behavior of distributions when data are subject to multi-dimensional clustering.

We apply our theoretical framework to the empirical analysis of growth-at-risk.
While our contribution does not concern identification or estimation per se, but rather the construction of valid standard errors, it is nevertheless important for empirical research to assess whether existing findings are robust to appropriate inference procedures.
We show that, even after accounting for two-way clustering and extreme quantiles, the main empirical conclusions in the existing literature remain qualitatively unchanged.

\bibliographystyle{ecta}
\bibliography{bibliography}

\appendix
\section*{Appendix}
\section{Proof of Theorem \ref{theorem:main}}\label{sec:theorem:main}

\begin{proof}
From Lemmas \ref{lemma:w} and \ref{lemma:g} in Appendix \ref{sec:lemma:w} and \ref{sec:lemma:g}, respectively, we have
\begin{equation*}
Q_{NT}(z, \tau) \stackrel{d}{\rightarrow}
\tilde W z+\frac{1}{2} \cdot\left(\frac{m^{-\xi}-1}{-\xi}\right) \cdot z^2,
\end{equation*}
where $\tilde W$ follows the standard normal distribution.

Hence, the limit of $Q_{NT}(z, \tau)$ is minimized at 
\[
\left(\frac{\xi}{m^{-\xi}-1}\right)  W\sim N\left(0, \frac{\xi^2}{\left(m^{-\xi}-1\right)^2}  \right).
\]
By the convexity lemma \citep{geyer1996asymptotics,knight1999epi}
therefore, it follows that
\begin{equation*}
\widehat{Z}_{NT} \stackrel{d}{\rightarrow} N\left(0, \frac{\xi^2}{\left(m^{-\xi}-1\right)^2} \right),
\end{equation*}
as desired.
\end{proof}

\subsection{Auxiliary Lemma: Asymptotic Behavior of $\sigma_{NT}^{-1}W_{NT}(\tau)$}\label{sec:lemma:w}

\begin{lemma}\label{lemma:w}
If Assumptions \ref{a:b} and \ref{a:lyapunov} are satisfied, then 
\[
\frac{W_{NT}(\tau)}{\sigma_{NT}}\xrightarrow{d} N(0, 1).
\]
\end{lemma}

\begin{proof}
Decompose $W_{NT}(\tau)$ as
\begin{equation*}
W_{NT}(\tau) = A_{1,\tau}+A_{2,\tau}+A_{3,\tau}+A_{4,\tau},
\end{equation*}
where
\begin{align*}
A_{1,\tau} =&\frac{-\sqrt{T}}{\sqrt{\tau N}}\sum_{i=1}^N E[\left(\tau-\mathbbm{1}\left[Y_{it}< \beta(\tau)\right]\right) |\alpha_i],\\
A_{2,\tau} =&\frac{-\sqrt{N}}{\sqrt{\tau T}}\sum_{t=1}^T E[\left(\tau-\mathbbm{1}\left[Y_{it}< \beta(\tau)\right]\right) |\gamma_{t}],\\
A_{3,\tau} =&\frac{-1}{\sqrt{\tau N T}}\sum_{i=1}^N \sum_{t=1}^T (E[\left(\tau-\mathbbm{1}\left[Y_{it}< \beta(\tau)\right]\right) |\alpha_i, \gamma_t]-E[\left(\tau-\mathbbm{1}\left[Y_{it}< \beta(\tau)\right]\right) |\alpha_i]\\
&\:\:-E[\left(\tau-\mathbbm{1}\left[Y_{it}< \beta(\tau)\right]\right) |\gamma_t]),\\
A_{4,\tau}=&\frac{-1}{\sqrt{\tau NT}}\sum_{i=1}^N \sum_{t=1}^T(\left(\tau-\mathbbm{1}\left[Y_{it}< \beta(\tau)\right]\right)-E[\left(\tau-\mathbbm{1}\left[Y_{it}< \beta(\tau)\right]\right) |\alpha_i, \gamma_t]).
\end{align*}

First, we show that under Assumption \ref{a:density}, the variance of $A_{3,\tau}$ becomes negligible compared to the variance of $A_{4,\tau}$.  For shorthand notation define $D_{it}(\tau)$ and $\omega_{it}$ as
$$
D_{it}(\tau) = (\tau-\mathbbm{1}\left[Y_{it}<\beta(\tau)\right]),
$$ 
\[
\omega_{it}=E[D_{it}(\tau)|\alpha_i,\gamma_t]-E[D_{it}(\tau)|\alpha_i]-E[D_{it}(\tau)|\gamma_t].
\]
To Calculate the variance of $A_{3,\tau}$, first observe that 
\[
E[\omega_{it}]=E[\omega_{it}|\alpha_i]=E[\omega_{it}|\gamma_t]=-E[D_{it}(\tau)]=0.
\]
If $i\neq i'$ and $t\neq t'$, then we have
$$
E[\omega_{it}\cdot \omega_{i't'}]=(E[\omega_{it}])^2=0,
$$
since $\alpha_i,\alpha_{i'},\gamma_t,\gamma_{t'}$ are independent.
Furthermore, we can show that 
\begin{align*}
E[\omega_{it}\cdot \omega_{it'}]=E [E[\omega_{it}\cdot  \omega_{it'}|\alpha_i]]=
E [E[\omega_{it}|\alpha_i]E [\omega_{it'}|\alpha_i]]=0,
\\
E[\omega_{it}\cdot \omega_{i't}]=E [E[\omega_{it}\cdot  \omega_{i't}|\gamma_t]]=
E [E[\omega_{it}|\gamma_t]E [\omega_{i't}|\gamma_t]]=0.
\end{align*}
Hence, the variance of $A_{3,\tau}$ satisfies 
\begin{align*}
\Var(A_{3,\tau})&=\frac{1}{\tau N T}\sum_{i=1}^N\sum_{t=1}^T E[\omega_{it}^2]\\
&=\frac{1}{\tau N T} \sum_{i=1}^N (E[\left(\tau-\mathbbm{1}\left[Y_{it}< \beta(\tau)\right]\right) |\alpha_i, \gamma_t]-E[\left(\tau-\mathbbm{1}\left[Y_{it}< \beta(\tau)\right]\right) |\alpha_i]\\
&\:\:\:-E[\left(\tau-\mathbbm{1}\left[Y_{it}< \beta(\tau)\right]\right) |\gamma_t])^2\\
&=\frac{1}{\tau}\cdot O(\tau^2\vee E [P^2\left(Y_{it}<\beta(\tau)|\alpha_i\right)]\vee E [P^2\left(Y_{it}<\beta(\tau)|\gamma_t\right)]\vee E [P^2\left(Y_{it}<\beta(\tau)|\alpha_i, \gamma_t\right)]).
\end{align*}
From Assumption \ref{a:b}, it follows that
\begin{align*}
\Var(A_{3,\tau})=& \frac{1}{\tau}\cdot O(\tau^2\vee E [P^2\left(Y_{it}<\beta(\tau)|\alpha_i\right)]\vee E [P^2\left(Y_{it}<\beta(\tau)|\gamma_t\right)]\vee E [P^2\left(Y_{it}<\beta(\tau)|\alpha_i, \gamma_t\right)])\\
= &\frac{1}{\tau}\cdot o(\tau)\\
= & o(1).
\end{align*}
In order to calculate the variance of $A_{4,\tau}$, define $\psi_{it}$ as
\begin{align*}\psi_{it}=&
\left(\tau-\mathbbm{1}\left[Y_{it}<\beta(\tau)\right]\right)-E[\left(\tau-\mathbbm{1}\left[Y_{it}<\beta(\tau)\right]\right) X_{it}|\alpha_i,\gamma_t]\\
=& D_{it}(\tau)-E[D_{it}(\tau)|\alpha_i,\gamma_t]).
\end{align*}
Note that $\psi_{it}$ and $\psi_{i't'}$ are independent given $\alpha_i$, $\alpha_{i'}$, $\gamma_t$ and $\gamma_{t'}$.
By the law of total variance, we can decompose the variance of $A_{4,\tau}$ as
\[
\Var(A_{4,\tau})=E[ \Var(A_{4,\tau}| \alpha ,\gamma)]+ \Var[E[A_{4,\tau}| \alpha,\gamma]],
\]
where $\alpha=(\alpha_1,\alpha_2,....,\alpha_N)$ and 
$\gamma=(\gamma_1,\gamma_2,....,\gamma_T)$.
For the first term on the right-hand side, 
\[
E[ \Var(A_{4,\tau}| \alpha ,\gamma)]=E[\sum_{i=1}^N\sum_{t=1}^T E[\psi_{it}^2 |  \alpha_i ,\gamma_t]]=E[\frac{1}{NT\tau}\sum_{i=1}^N\sum_{t=1}^T \psi_{it}^2].
\]
For the second term, 
\[
E[A_{4,\tau}| \alpha,\gamma]=0,
\]
and hence,
\[
\Var[E[A_{4,\tau}| \alpha,\gamma]]=0.
\]
Therefore, we obtain
\begin{align*}
\Var(A_{4,\tau})=&
E\left[\frac{1}{NT\tau}\sum_{i=1}^N\sum_{t=1}^T \psi_{it}^2\right]\\
=&E[(D_{it}(\tau)-E[D_{it}(\tau)|\alpha_i,\gamma_t])^2]\\
=&\frac{1}{\tau}(\tau-2E[\mathbbm{1}\left[Y_{it}< \beta(\tau)\right]P\left(Y_{it}<\beta(\tau)|\alpha_i,\gamma_t\right)]
+E [P^2\left(Y_{it}<\beta(\tau)|\alpha_i,\gamma_t\right)]).
\end{align*}
It is worth noting that
\begin{align*}
E[\mathbbm{1}\left[Y_{it}< \beta(\tau)\right]P\left(Y_{it}<\beta(\tau)|\alpha_i,\gamma_t\right)=&E[E[\mathbbm{1}\left[Y_{it}< \beta(\tau)\right]P\left(Y_{it}<\beta(\tau)|\alpha_i,\gamma_t\right)|\alpha_i, \gamma_t]]\\
=& E [P^2\left(Y_{it}<\beta(\tau)|\alpha_i,\gamma_t\right)].
\end{align*}
From Assumption \ref{a:b} , 
\[
\lim_{\tau\to 0,\:\:\: N,T\to \infty} \Var(A_{4,\tau})=1.
\]
Hence,
\[
\lim_{\tau\to 0,\:\:\:N,T\to \infty} (\Var(A_{3,\tau})+\Var(A_{4,\tau}))\:\:= \lim_{\tau\to 0,\:\: N,T\to \infty} (\Var(A_{4,\tau}))=1.
\]
This leads to
\begin{equation}\label{124}
\lim_{\tau\to 0,\:\: N,T\to \infty}  \Var(W_{NT}(\tau)) \:\:= \lim_{\tau\to 0,\:\: N,T\to \infty}(\Var(A_{1,\tau})+\Var(A_{2,\tau})+\Var(A_{4,\tau})).
\end{equation}

Now, we show that $A_{1,\tau}/\Var(A_{1,\tau})$, $A_{2,\tau}/\Var(A_{2,\tau})$, and $A_{4,\tau}/\Var(A_{4,\tau})$ are all asymptotically normal. 
We start by showing asymptotic normality of $A_{1,\tau}/\Var(A_{1,\tau})$. Define $Z_{N,i}$
\[
Z_{N,i}= \frac{A_{1,\tau}}{\sqrt{T}}=\frac{1}{\sqrt{N\tau}}E[\left(\tau-\mathbbm{1}\left[Y_{it}< \beta(\tau)\right]\right) |\alpha_i].
\]
Note that $Z_{N,1},...,Z_{N,N}$ are independent as $\alpha_1,...,\alpha_N$ are independent. Hence, 
\[
\Var(\sum_{i=1}^N Z_{N,i})=\frac{\Var(A_{1,\tau})}{T}.
\]
From Assumption \ref{a:lyapunov} (i), the following Lyapunov condition holds.
\[
\frac{E[\sum_{i=1}^N Z_{N,i}^3]}{E[\sum_{i=1}^N Z_{N,i}^2]^{\frac{3}{2}}}=o(1).
\]
Therefore, Lindeberg-Feller CLT yields
\[
\frac{\sum_{i=1}^N Z_{N,i}}{\sqrt{\Var(A_{1,\tau})/T}}=\frac{A_{1,\tau}}{\sqrt{\Var(A_{1,\tau})}}\xrightarrow{d} N(0,1).
\]
Similarly for $A_{2,\tau}$, Define $W_{T,t}$ as
\[
W_{T,t}=\frac{1}{\sqrt{T\tau}} E[\left(\tau-\mathbbm{1}\left[Y_{it}< \beta(\tau)\right]\right) |\gamma_t].
\]
Note that $W_{T,1},...,W_{T,T}$ are independent as $\gamma_1,...,\gamma_T$ are. 
By Assumption \ref{a:lyapunov} (ii), we have
\[
\frac{E[\sum_{t=1}^T W_{T,t}^3]}{E[\sum_{t=1}^T W_{T,t}^2]^{\frac{3}{2}}}=o(1).
\]
Therefore, Lindeberg-Feller CLT yields
\[
\frac{\sum_{t=1}^T W_{T,t}}{\sqrt{\Var(A_{2,\tau})/N}}=\frac{A_{2,\tau}}{\sqrt{\Var(A_{2,\tau})}}\xrightarrow{d} N(0,1).
\]
For $A_{4,\tau}/\Var(A_{4,\tau})$, define $U_{N+T,it}$ as 
\[
U_{N+T,it}=\frac{1}{\sqrt{NT\tau}}(\left(\tau-\mathbbm{1}\left[Y_{it}< \beta(\tau)\right]\right)-E[\left(\tau-\mathbbm{1}\left[Y_{it}< \beta(\tau)\right]\right) |\alpha_i, \gamma_t]).
\]
Note that $U_{N+T,11},...,U_{N+T,NT}$ are independent conditional on $\alpha$ and $\gamma$.
By Assumption \ref{a:lyapunov} (iii), we have
\[
\frac{E[\sum_{i=1}^N \sum_{i=t}^T U_{N+T,it}^3|\alpha,\gamma]}{E[\sum_{i=1}^N \sum_{i=t}^T U_{N+T,it}^2|\alpha,\gamma]^\frac{3}{2}}=o(1).
\]
Therefore, given $\alpha$ and $\gamma$, Lindeberg-Feller CLT yields
\[
\frac{\sum_{i=1}^T U_{N+T}}{\Var(A_{4,\tau})}\xrightarrow{d} N(0,1).
\]
Since
\[
\Var(A_{4,\tau})=E[\Var(A_{4,\tau}|\alpha,\gamma)],
\]
we have
\[
\Var(A_{4,\tau}|\alpha,\gamma)=\frac{1}{NT\tau}\sum_{i=1}^N\sum_{t=1}^T E[\psi_{it}^2 |  \alpha_i ,\gamma_t]
\xrightarrow{p} E[\Var(A_{4,\tau}|\alpha,\gamma)]=\Var(A_{4,\tau}).
\]
Since $A_{1,\tau}$ and $A_{2,\tau}$ are independent, $\displaystyle\frac{A_{1,\tau}+A_{2,\tau}}{\sqrt(\Var(A_{1,\tau})+\Var(A_{2,\tau})}$ converges to the standard normal distribution.
By Chen and Rao (2007) Theorem 2, we obtain 
\begin{equation}\label{AN}
W_{NT}(\tau)/\sqrt{\Var(A_{1,\tau})+\Var(A_{2,\tau})+\Var(A_{4,\tau})}\xrightarrow{d} N(0,1).
\end{equation}
Combining \eqref{124} and \eqref{AN} we have
\[
\frac{W_{NT}(\tau)}{\sigma_{NT}}\xrightarrow{d} N(0,1).
\]
This completes the proof.
\end{proof}

\subsection{Auxiliary Lemma: Asymptotic Behavior of $G_{NT}(z,\tau)$}\label{sec:lemma:g}

\begin{lemma}\label{lemma:g}
If Assumptions \ref{a:regular}, \ref{a:b}, and \ref{a:density} are satisfied, then
\[
G_{NT}(z, \tau)\xrightarrow{p} E[G_{NT}(z, \tau)]=  \frac{1}{2} \cdot \frac{m^{-\xi}-1}{-\xi} \cdot  z^2.
\]
\end{lemma}

\begin{proof}
Rewrite $G_{NT}(z, \tau)$ as
\begin{align*}
G_{NT}(z, \tau)&= \sum_{i=1}^N \sum_{t=1}^T \frac{a_{NT}}{\sigma^2_{NT}}\cdot\left(\int_0^{\sigma_{NT}\cdot z / a_{NT}}\left[\frac{\mathbbm{1}\left(Y_{it}-\beta(\tau) \leq s\right)-\mathbbm{1}\left(Y_{it}-\beta(\tau) \leq 0\right)}{\sqrt{\tau NT}}\right] d s\right) \\
&=\frac{1}{\sigma_{NT}}\sum_{i=1}^N \sum_{t=1}^T\left(\int_0^{z}\left[\frac{\mathbbm{1}\left(Y_{it}-\beta(\tau) \leq \sigma_{NT} \cdot s / a_{NT}\right)-\mathbbm{1}\left(Y_{it}-\beta(\tau) \leq 0\right)}{\sqrt{\tau NT}}\right] d s\right),
\end{align*}
For computing $E[G_{NT}(z, \tau)]$,  we follow the proof outlined in step 2 of Theorem 5.1 in Chernozhukov (2005).
\begin{equation}\label{eq:c}
\begin{aligned}
E G_{NT}(z, \tau) &=\frac{NT}{\sigma_{NT}}  \cdot \left(\int_0^{z} \frac{F_Y\left[F_Y^{-1}(\tau)+\sigma_{NT}\cdot s / a_{NT}\right]-F_Y\left[F_Y^{-1}(\tau)\right]}{\sqrt{\tau NT}} d s\right) \\
& \stackrel{(1)}{=} NT \cdot \left(\int_0^{z} \frac{f_Y\left\{F_Y^{-1}(\tau)+o\left(F_Y^{-1}(m \tau)-F_Y^{-1}(\tau)\right)\right\}}{a_{NT} \cdot \sqrt{\tau NT}} \cdot s \cdot d s\right) \\
& \stackrel{(2)}{\sim}  NT\cdot \left(\int_0^{z} \frac{f_Y\left\{F_Y^{-1}(\tau)\right\}}{a_{NT} \cdot \sqrt{\tau NT}} \cdot s \cdot d s\right) \\
&=NT  \cdot \left(\frac{1}{2} \cdot\left(z\right)^2 \cdot \frac{f_Y\left\{F_Y^{-1}(\tau)\right\}}{a_{NT} \cdot \sqrt{\tau NT}}\right) \\
&=E\left(\frac{1}{2} \cdot\left(z\right)^2 \cdot \frac{F_Y^{-1}(m \tau)-F_Y^{-1}(\tau)}{\tau\left(f_Y\left\{F_Y^{-1}(\tau)\right\}\right)^{-1}}\right) \\
& \stackrel{(3)}{\sim} E\left(\frac{1}{2} \cdot\left(z\right)^2 \cdot \frac{m^{-\xi}-1}{-\xi}\right) \\
& = \frac{1}{2} \cdot \frac{m^{-\xi}-1}{-\xi} \cdot z^2 .
\end{aligned}
\end{equation}

Equality (1) is by the definition of $a_{NT}$ and a Taylor expansion. Since Assumption \ref{a:b} ensures that $\displaystyle \sigma_{NT}^{-1} {\sqrt{\tau N T}} \rightarrow\infty$ uniformly over $s$ in any compact subset of $\mathbb{R}$, we have
\[
\sigma_{NT}\cdot s / a_{NT}=\sigma_{NT}\cdot s \cdot\left(F_Y^{-1}(m \tau)-F_Y^{-1}(\tau)\right) / \sqrt{\tau NT}=o\left(F_Y^{-1}(m \tau)-F_Y^{-1}(\tau)\right).
\]
To show equivalence (2), it suffices to prove that, for any sequence $v_\tau=$ $o\left(F_Y^{-1}(m \tau)-F_Y^{-1}(\tau)\right)$ with $m>1$ as $\tau \to 0$,
$f_Y\left(F_Y^{-1}(\tau)+v_\tau\right) \sim f_Y\left(F_Y^{-1}(\tau)\right)$ holds.
By Assumption \ref{a:regular}, we have
\[f_Y\left(F_Y^{-1}(l \tau)\right) \sim l^{\xi+1} f_Y\left(F_Y^{-1}(\tau)\right).
\]
Locally uniform in $l$,
\[
f_Y\left(F_Y^{-1}(\tau)+\left[F_Y^{-1}(l \tau)-F_Y^{-1}(\tau)\right]\right) \sim l^{\xi+1} f_Y\left(F_Y^{-1}(\tau)\right).
\]
Hence, for any $l_\tau \rightarrow 1$,
\[
f_Y\left(F_Y^{-1}(\tau)+\left[F_Y^{-1}\left(l_\tau \tau\right)-F_Y^{-1}(\tau)\right]\right) \sim f_Y\left(F_Y^{-1}(\tau)\right).
\]
It follows that, for any sequence $v_\tau=o\left(\left[F^{-1}(m \tau)-F^{-1}(\tau)\right]\right)$ with $m>1$ as $\tau \to 0$,
\[
f_Y\left(F^{-1}(\tau)+v_\tau\right) \sim f_Y\left(F_Y^{-1}(\tau)\right).
\]
The equivalence (3) can be shown by the regular variation property, 
\[
\begin{aligned}
\frac{F_Y^{-1}(m \tau)-F_Y^{-1}(\tau)}{\tau\left(f_Y\left[F_Y^{-1}(\tau)\right]\right)^{-1}} & \equiv \int_1^m \frac{f_Y\left[F_Y^{-1}(\tau)\right]}{f_Y\left[F_Y^{-1}(s \tau)\right]} d s \\
& \sim \int_1^m s^{-\xi-1} d s \\
& =\frac{m^{-\xi}-1}{-\xi} \quad(\ln m \text { if } \xi=0) .
\end{aligned}
\]
Now, it remains to show
\[
G_{NT}(z, \tau)\longrightarrow E[G_{NT}(z, \tau)].
\]
Let us introduce the short-hand notation:
\begin{equation*}
\lambda_{it}=\frac{1}{\sigma_{NT}}\int_0^{z}\left[\mathbbm{1}\left(Y_{it}-\beta(\tau) \leq \sigma_{NT}\cdot  s / a_{NT}\right)-\mathbbm{1}\left(Y_{it}-\beta(\tau) \leq 0\right)\right] d s .
\end{equation*}
Then, we can write
\begin{align*}
 G_{NT}(z, \tau)=\frac{1}{\sqrt{NT\tau}}\sum_{i=1}^N\sum_{t=1}^T \lambda_{it}&=B_1+B_2+B_3+B_4
\end{align*}
where
\begin{align*}
B_1=&\frac{-\sqrt{T}}{\sqrt{\tau N}}\sum_{i=1}^N E[\lambda_{it}|\alpha_i],\\
B_2=&\frac{-\sqrt{N}}{\sqrt{\tau T}}\sum_{t=1}^T E[\lambda_{it}|\gamma_{t}],\\
B_3=&\frac{-1}{\sqrt{\tau N T}}\sum_{i=1}^N \sum_{t=1}^T (E[\lambda_{it}|\alpha_i, \gamma_t]-E[\lambda_{it}|\alpha_i]-E[\lambda_{it}|\gamma_t]),\\
B_4=&\frac{-1}{\sqrt{\tau NT}}\sum_{i=1}^N \sum_{t=1}^T(\lambda_{it}-E[\lambda_{it}|\alpha_i, \gamma_t]).
\end{align*}
We first focus on the $B_1$ term.
\begin{align*}
\frac{\sqrt{ NT}\cdot}{\sqrt{\tau}}\cdot E[\lambda_{it}|\alpha_i]
&= \frac{\sqrt{ NT}\cdot}{\sqrt{\tau} \sigma_{NT}} \cdot E\left[\int_0^{z}[\mathbbm{1}\left(Y_{it}-\beta(\tau) \leq \sigma_{NT}\cdot  s / a_{NT}\right)-\mathbbm{1}\left(Y_{it}-\beta(\tau) \leq 0\right)]d s\middle|\alpha_i\right]\\
&=
\frac{\sqrt{ NT}\cdot}{\sqrt{\tau} \sigma_{NT}}  \cdot \left(\int_0^{z} F_Y\left[F_Y^{-1}(\tau)+\sigma_{NT}\cdot s / a_{NT} \middle | \alpha_i\right]-F_Y\left[F_Y^{-1}(\tau)| \alpha_i\right] d s\right) \\
& \stackrel{(1)}{\sim}  \left(\int_0^{z} \frac{f_Y\left\{F_Y^{-1}(\tau) | \alpha_i \right\}}{a_{NT} \cdot } \cdot s \cdot d s\right) \\
&= \left(\frac{1}{2} \cdot\left(z\right)^2 \cdot \frac{f_Y\left\{F_Y^{-1}(\tau) | \alpha_i\right\}}{a_{NT}}  \right) \\
&=\left(\frac{1}{2} \cdot\left(z\right)^2 \cdot \frac{F_Y^{-1}(m \tau)-F_Y^{-1}(\tau)}{\tau\left(f_Y\left\{F_Y^{-1}(\tau)\right | \alpha_i\}\right)^{-1}}\right) \\
& \stackrel{(2)}{\sim} \left(\frac{1}{2} \cdot\left(z\right)^2 \cdot \frac{m^{-\xi}-1}{-\xi}\right) \\
& \equiv \frac{1}{2} \cdot \frac{m^{-\xi}-1}{-\xi} \cdot z^2 .
\end{align*}
Here,
equivalence (1) follows from Assumption \ref{a:regular} (ii) and the same arguments made in Equation \eqref{eq:c}.
Equivalence (2) follows from Assumption \ref{a:regular} (ii) and Assumption \ref{a:density} (i), as
\[
\begin{aligned}
\frac{F_Y^{-1}(m \tau)-F_Y^{-1}(\tau)}{\tau\left(f_Y\left[F_Y^{-1}(\tau)| \alpha_i\right]\right)^{-1}} & \equiv \int_1^m \frac{f_Y\left[F_Y^{-1}(\tau)\right]}{f_Y\left[F_Y^{-1}(s \tau)\right]} \cdot \frac{f_Y\left[F_Y^{-1}(\tau)| \alpha_i\right]}{f_Y\left[F_Y^{-1}(\tau)\right]}d s \\
& \sim \int_1^m s^{-\xi-1} d s \\
& =\frac{m^{-\xi}-1}{-\xi} \quad(\ln m \text { if } \xi=0) .
\end{aligned}
\]
It follows that
\begin{align*}
\Var(E[\lambda_{it}|\alpha_i])=&O(E[E[\lambda_{it}|\alpha_i]^2])=
O(\tau/NT)
\end{align*}
which leads to 
\[
\Var(B_1)=O(\tau/NT)\cdot \frac{T}{\tau N}\cdot N =O(1/{N}).
\]
Similarly, we obtain 
\begin{align*}
\Var(E[\lambda_{it}|\gamma_t])=&O(E[E[\lambda_{it}|\gamma_t]^2])=
O(\tau/NT)\quad \text{and}\\
\Var(E[\lambda_{it}|\alpha_i,\gamma_t])=&O(E[E[\lambda_{it}|\alpha_i,\gamma_t]^2])=O(\tau/NT).
\end{align*}
Hence,
\begin{align*}
\Var(B_2)=&O(\tau/NT)\cdot \frac{N}{\tau T}\cdot T =O(1/{T})\quad
\text{and}\\
\Var(B_3)=&O(\tau/NT)\cdot \frac{1}{\tau N T}\cdot T =O(1/{NT})
\end{align*}
follow.
For $B_4$, recall that 
\begin{equation*}
\lambda_{it}=\frac{1}{\sigma_{NT}}\int_0^{z}\left[\mathbbm{1}\left(Y_{it}-\beta(\tau) \leq \sigma_{NT}\cdot  s / a_{NT}\right)-\mathbbm{1}\left(Y_{it}-\beta(\tau) \leq 0\right)\right] d s .
\end{equation*}
Let
\begin{equation*}
\mu_{it}=\frac{1}{\sigma_{NT}}\left(\mathbbm{1}\left(Y_{it}-\beta(\tau) \leq \sigma_{NT}\cdot z / a_{NT}\right)-\mathbbm{1}\left(Y_{it}-\beta(\tau) \leq 0\right)\right).
\end{equation*}
Then, there exists $K_0>0$ such that 
\begin{equation}\label{bound}
\left|\lambda_{it}\right| \leq K_0\left|\mu_{it}\right|.
\end{equation}
But then,
\[
\Var(\lambda_{it})=O(E[\lambda^2_{it}])=O(E[\mu^2_{it}])=O(E[\abs{\mu_{it}}])
=O(\sqrt{\tau/NT})
\]
where the last equation follows from equation (\ref{eq:c}).
Therefore,
\[
\Var(B_4)=\frac{1}{NT\tau} \cdot NT\cdot O(\sqrt{\tau / NT})=O({1/\sqrt{NT\tau}}).
\]
From the all the above convergence results, we have
\[\Var(G_{NT}(z, \tau))=\left(O\left(\frac{1}{N}\right)+O\left(\frac{1}{T}\right)+O\left(\frac{1}{NT}\right)+O\left(\frac{1}{\sqrt{NT\tau}}\right)\right),
\]
and this completes the proof.
\end{proof}

\section{Proof of Theorem \ref{theorem:regression}}\label{sec:theorem:regression}
\begin{proof}
From Lemmas \ref{W2} and \ref{G2},
\begin{equation}
Q_{NT}(z, \tau) \rightarrow 
W^{\prime} z+\frac{1}{2} \cdot\left(\frac{m^{-\xi}-1}{-\xi}\right) \cdot z^{\prime} \mathcal{Q}_H z,
\end{equation}
where 
\begin{align*}
&W\sim N(0,\Sigma),
\\
&G_{NT}
\stackrel{p}{\rightarrow} \frac{1}{2}\left(\frac{m^{-\xi}-1}{-\xi}\right) z^{\prime} \mathcal{Q}_{H} z.
\end{align*}
Hence, the limit of $Q_{NT}(z, \tau)$ is minimized at 
\[
\left(\frac{\xi}{m^{-\xi}-1}\right) \mathcal{Q}_H^{-1} W\sim N\left(0, \frac{\xi^2}{\left(m^{-\xi}-1\right)^2} \mathcal{Q}_H^{-1} \Sigma \mathcal{Q}_H^{-1}\right).
\]
By the convexity lemma \citep{geyer1996asymptotics,knight1999epi}, we have 
\begin{equation*}
\tilde{Z}_{NT} \stackrel{d}{\rightarrow} N\left(0, \frac{\xi^2}{\left(m^{-\xi}-1\right)^2} \mathcal{Q}_H^{-1} \Gamma \mathcal{Q}_H^{-1}\right).
\end{equation*}
Since  $\widehat Z_{NT}$ satisfies
\[
\widehat Z_{NT}=c_{NT} \cdot \Sigma^{-\frac{1}{2}}_{NT}\cdot\mathcal{Q}_H\cdot   \tilde Z_{NT},
\]
we have
\begin{equation*}
\widehat{Z}_{NT} \stackrel{d}{\rightarrow} N\left(0, \frac{\xi^2}{\left(m^{-\xi}-1\right)^2} I \right)
\end{equation*}
and this completes a proof.

\end{proof}
\subsection{Auxiliary Lemma: Asymptotic Behavior of $W_{NT}(z,\tau)$ (Quantile Regresion case)}\label{sec:lemma:regression:w}

\begin{lemma}\label{W2}

If Assumption \ref{b2} and \ref{lyapunov2} are true then 
\[
c^{-1}_{NT}\cdot W_{NT}(\tau)\xrightarrow{d} N(0, \Gamma)
\]
where 
\[
\Gamma=\lim_{N,T\to \infty, \tau\to 0 } c^{-2}_{NT}\cdot \Var(W_{NT}).
\]

\end{lemma}

\begin{proof}
We decompose
\begin{equation}
W_{NT}(\tau) \equiv \frac{-1}{\sqrt{\tau NT}} \sum_{i=1}^N\sum_{t=1}^T\left(\tau-\mathbbm{1}\left[Y_{it}<X_{it}^{\prime} \beta(\tau)\right]\right) X_{it}
\end{equation}
as
\begin{equation}
\begin{aligned}
W_{NT}(\tau) =&\frac{-\sqrt{T}}{\sqrt{\tau N}}\sum_{i=1}^N E[\left(\tau-\mathbbm{1}\left[Y_{it}<X_{it}^{\prime} \beta(\tau)\right]\right) X_{it}|\alpha_i]\\
&+\frac{-\sqrt{N}}{\sqrt{\tau T}}\sum_{t=1}^T E[\left(\tau-\mathbbm{1}\left[Y_{it}<X_{it}^{\prime} \beta(\tau)\right]\right) X_{it}|\gamma_{t}]\\
&+\frac{-1}{\sqrt{\tau N T}}\sum_{i=1}^N \sum_{t=1}^T (E[\left(\tau-\mathbbm{1}\left[Y_{it}<X_{it}^{\prime} \beta(\tau)\right]\right) X_{it}|\alpha_i, \gamma_t]-E[\left(\tau-\mathbbm{1}\left[Y_{it}<X_{it}^{\prime} \beta(\tau)\right]\right) X_{it}|\alpha_i]\\
&\:\:-E[\left(\tau-\mathbbm{1}\left[Y_{it}<X_{it}^{\prime} \beta(\tau)\right]\right) X_{it}|\gamma_t])\\
&+\frac{-1}{\sqrt{\tau NT}}\sum_{i=1}^N \sum_{t=1}^T(\left(\tau-\mathbbm{1}\left[Y_{it}<X_{it}^{\prime} \beta(\tau)\right]\right)X_{it}-E[\left(\tau-\mathbbm{1}\left[Y_{it}<X_{it}^{\prime} \beta(\tau)\right]\right) X_{it}|\alpha_i, \gamma_t])\\
&=A_{1,\tau}+A_{2,\tau}+A_{3,\tau}+A_{4,\tau}.
\end{aligned}
\end{equation}
First, we show that under Assumptions \ref{b2}, variance of $A_{3,\tau}$ becomes negligible compared to the variance of $A_{4,\tau}$.  For short hand notation define $D_{i t}(\tau)$ and $\omega_{it}$ as
\[D_{i t}(\tau)=\tau-1\left[Y_{i t}<X_{it}^{\prime}\beta(\tau)\right]\]
\[\omega_{it}=
E[\left(\tau-\mathbbm{1}\left[Y_{it}<X_{it}^{\prime} \beta(\tau)\right]\right) X_{it}|\alpha_i, \gamma_t]-E[\left(\tau-\mathbbm{1}\left[Y_{it}<X_{it}^{\prime} \beta(\tau)\right]\right) X_{it}|\alpha_i]\\
-E[\left(\tau-\mathbbm{1}\left[Y_{it}<X_{it}^{\prime} \beta(\tau)\right]\right) X_{it}|\gamma_t].
\]
Note that
\begin{align*}
&\omega_{it}=E[D_{it}|\alpha_i,\gamma_t]-E[D_{it}|\alpha_i]-E[D_{it}|\gamma_t]
\text{ and }\\
&E[\omega_{it}]=E[\omega_{it}|\alpha_i]=E[\omega_{it}|\gamma_t]=-E[D_{it}]=0.
\end{align*}
If $i\neq i'$ and  $t\neq t'$, then
$E[\omega_{it}\cdot \omega_{i't'}]=E[\omega_{it}]E[\omega_{i't'}]=0$,\\
 since $\alpha_i,\alpha_{i'},\gamma_t,\gamma_{t'}$ are independent.
For the correlation between $\omega_{it}$ and $\omega_{it'}$, 
\[
E[\omega_{it}\cdot \omega_{it'}]=E [E[\omega_{it} \omega_{it'}|\alpha_i]]=
E [E[\omega_{it}|\alpha_i]E [\omega_{it'}|\alpha_i]]=0
\]
\[
E[\omega_{it}\cdot \omega_{i't}]=E [E[\omega_{it} \omega_{i't}|\gamma_t]]=
E [E[\omega_{it}|\alpha_i]E [\omega_{i't}|\gamma_t]]=0.
\]
The variance of $A_{3,\tau}$ satisfies 
\begin{equation}\label{varA3x}
\begin{aligned}
\Var(A_{3,\tau})&=\frac{1}{\tau N T}\sum_{i=1}^N\sum_{t=1}^T E[\omega_{it} \omega'_{it}]\\
&=
\frac{1}{\tau}\cdot O(E[E[D_{it}(\tau)| \alpha_{i}]E[D^{\prime}_{it}| \alpha_{i}]]) + O( E[E[D_{it}(\tau)| \gamma_t]E[D^{\prime}_{it}| \gamma_t]])+O(E[E[D_{it}(\tau)| \alpha_{i}, \gamma_t]]E[D^{\prime}_{it}| \alpha_{i}, \gamma_t]]) \\
&= o(1).
\end{aligned}
\end{equation}
The last equation follows from Assumption \ref{b2}.
For the variance of $A_{4,\tau}$, define $\psi_{it}$ as
\[ \psi_{it}=
\left(\tau-\mathbbm{1}\left[Y_{it}<X_{it}^{\prime} \beta(\tau)\right]\right)X_{it}-E[\left(\tau-\mathbbm{1}\left[Y_{it}<X_{it}^{\prime} \beta(\tau)\right]\right) X_{it}|\alpha_i, \gamma_t]
\]
Note that $\psi_{it}$ and $\psi_{i't'}$ are independent given $\alpha_i$, $\alpha_{i'}$, $\gamma_t$ and $\gamma_{t'}$.
By the law of total variance,
\[
\Var(A_{4,\tau})=E[ \Var(A_{4,\tau}| \alpha ,\gamma)]+ \Var[E[A_{4,\tau}| \alpha ,\gamma]]
\]
where $\alpha=\alpha_1,\alpha_2,....,\alpha_N$
$\gamma=\gamma_1,\gamma_2,....,\gamma_T$.
For the second term,
\[E[A_{4,\tau}| \alpha ,\gamma]=0,\]
and hence,
\[
\Var[E[A_{4,\tau}| \alpha_i,\gamma_t]]=0.
\]
Since $\psi_{it}$ and $\psi_{i't'}$ are independent given $\alpha_i$, $\alpha_{i'}$, $\gamma_t$ and $\gamma_{t'}$, $E[ \Var(A_{4,\tau}| \alpha ,\gamma)]$ can be given as 
\[
E[ \Var(A_{4,\tau}| \alpha ,\gamma)]=E[E\sum_{i=1}^N\sum_{t=1}^T \psi_{it} \psi^\prime_{it}|  \alpha_i ,\gamma_t]]=E[\frac{1}{NT\tau}\sum_{i=1}^N\sum_{t=1}^T \psi_{it}\psi^\prime_{it}].
\]
Hence
\begin{equation*}
\begin{aligned}
\Var(A_{4,\tau})=&E[\frac{1}{NT\tau}\sum_{i=1}^N\sum_{t=1}^T \psi_{it}\psi^\prime_{it}]\\
=&\frac{1}{\tau}E[(D_{it}(\tau)-E[D_{it}(\tau)|\alpha_i,\gamma_t])E[(D^\prime_{it}(\tau)-E[D^\prime_{it}(\tau)|\alpha_i,\gamma_t])]]\\
=&\frac{1}{\tau} (E[D_{it}(\tau)D^\prime_{it}(\tau)]-E[D_{it}(\tau)E[D^\prime_{it}(\tau)|\alpha_i,\gamma_t]]-E[E[D^\prime_{it}(\tau)|\alpha_i,\gamma_t]D^\prime_{it}(\tau)]\\
&+E[E[D_{it}(\tau)|\alpha_i,\gamma_t]E[D^\prime_{it}(\tau)|\alpha_i,\gamma_t] ]).
\end{aligned}
\end{equation*}
It is worth noting that 
\begin{align*}
E[D_{it}(\tau)E[D^\prime_{it}(\tau)|\alpha_i,\gamma_t]]=E[E[D_{it}(\tau)E[D^\prime_{it}(\tau)|\alpha_i,\gamma_t]|\alpha_i,\gamma_t]]=E[E[D_{it}(\tau)|\alpha_i,\gamma_t]E[D^\prime_{it}(\tau)|\alpha_i,\gamma_t] ].
\end{align*}
Hence, from Assumption \ref{b2} (iii),
\begin{equation}\label{varA4x}
\Var(A_{4,\tau})\rightarrow E[X_{it} X_{it}^\prime].
\end{equation}
Combining \eqref{varA3x} and \eqref{varA4x}, we have
\[
\lim_{\tau\to 0,\:\:\:N,T\to \infty} (\Var(A_{3,\tau})+\Var(A_{4,\tau}))\:\:= \lim_{\tau\to 0,\:\: N,T\to \infty} (\Var(A_{4,\tau}))=1.
\]
This leads to
\begin{equation}\label{124x}
\lim_{\tau\to 0,\:\: N,T\to \infty}  \Var(W_{NT}(\tau)) \:\:= \lim_{\tau\to 0,\:\: N,T\to \infty}(\Var(A_{1,\tau})+\Var(A_{2,\tau})+\Var(A_{4,\tau})).
\end{equation}

Now, we show that $A_{1,\tau}/\Var(A_{1,\tau})$, $A_{2,\tau}/\Var(A_{2,\tau})$, and $A_{4,\tau}/\Var(A_{4,\tau})$ are all asymptotically normal. 
We start by showing asymptotic normality of $A_{1,\tau}/\Var(A_{1,\tau})$. Define $Z_{N,i}$
\[
Z_{N,i}= \frac{A_{1,\tau}}{\sqrt{T}}=\frac{1}{\sqrt{N\tau}}E[\left(\tau-\mathbbm{1}\left[Y_{it}<X_{it}^{\prime} \beta(\tau)\right]\right) X_{it}|\alpha_i].
\]
Note that $Z_{N,1},...,Z_{N,N}$ are independent as $\alpha_1,...,\alpha_N$ are independent. Hence, 
\[
\Var(\sum_{i=1}^N Z_{N,i})=\frac{\Var(A_{1,\tau})}{T}.
\]
From Assumption \ref{lyapunov2} (i), the following condition holds.
\[
\frac{E[\sum_{i=1}^N \norm {Z_{N,i}}^3]}{E[\sum_{i=1}^N \norm {Z_{N,i}}^2]^{\frac{3}{2}}}=o(1).
\]
Also note that 
\[
\Var\left(\sum_{i=1}^N Z_{N,i}\right)=\frac{\Var(A_{1,\tau})}{T}.
\]
Lindeberg-Feller CLT yields,
\[
{\Var(A_{1,\tau})}^{-\frac{1}{2}} A_{1,\tau}=
\Var\left(\sum_{i=1}^N Z_{N,i}\right)^{-\frac{1}{2}}\left(\sum_{i=1}^N Z_{N,i}\right)\xrightarrow{d} N(0,1)
\]
For$A_{2,\tau}$, let
\[
W_{T,t}=\frac{1}{\sqrt{T\tau}} E[\left(\tau-\mathbbm{1}\left[Y_{it}<X_{it}^{\prime} \beta(\tau)\right]\right) X_{it}|\gamma_t]
\]
$W_{T,1},...,W_{T,T}$ are independent. 
From Assumption \ref{lyapunov2} (ii)
\[
\frac{E[\sum_{i=1}^N \norm {W_{T,t}}^3]}{E[\sum_{i=1}^N \norm {W_{T,t}}^2]^{\frac{3}{2}}}=o(1)
\]
Similarly Lindeberg-Feller CLT yields,
\[
{\Var(A_{2,\tau})}^{-\frac{1}{2}} A_{2,\tau}=
\Var\left(\sum_{t=T}^N W_{T,t}\right)^{-\frac{1}{2}}\left(\sum_{t=1}^T W_{T,t}\right)\xrightarrow{d} N(0,1).
\]
For $A_{4,\tau}$, let 
\[
U_{N+T,it}=\frac{1}{\sqrt{NT\tau}}(\left(\tau-\mathbbm{1}\left[Y_{it}<X_{it}^{\prime} \beta(\tau)\right]\right)X_{it}-E[\left(\tau-\mathbbm{1}\left[Y_{it}<X_{it}^{\prime} \beta(\tau)\right]\right) X_{it}|\alpha_i, \gamma_t])
\]
$U_{N+T,11},...,U_{N+T,NT}$ are independent conditional on $\alpha$ and $\gamma$.
From Assumption \ref{lyapunov2} (iii)
\[
\frac{E[\sum_{i=1}^N \sum_{i=t}^T \norm {U_{N+T,it}}^3|\alpha,\gamma]}{E[\sum_{i=1}^N \sum_{i=t}^T \norm {U_{N+T,it}}^2|\alpha,\gamma]^\frac{3}{2}}=o(1).
\]
Lindeberg-Feller CLT yields, given $\alpha$ and $\gamma$, 
\[
{\Var(A_{4,\tau})}^{-\frac{1}{2}}\sum_{i=1}^T U_{N+T}\xrightarrow{d} N(0,1).
\]
Additionally, we can show that
\[
\Var(A_{4,\tau}|\alpha,\gamma)=\frac{1}{NT\tau}\sum_{i=1}^N\sum_{t=1}^T E[\psi_{it}\psi^\prime_{it} |  \alpha_i ,\gamma_t]
\xrightarrow{p} E[\Var(A_{4,\tau}|\alpha,\gamma)]=\Var(A_{4,\tau}).
\]
From Chen and Rao's (2007) Theorem 2, we obtain
\begin{equation}\label{ANx}
(\Var(A_{1,\tau})+\Var(A_{2,\tau})+\Var(A_{4,\tau}))^{-\frac{1}{2}} W_{NT}(\tau)\xrightarrow{d} N(0,1).
\end{equation}
Combining \eqref{124x} and \eqref{ANx} we have
\[
\Sigma^{-\frac{1}{2}}_{NT} W_{NT}\xrightarrow{d} N(0,I)
\]
where $I$ is an identity matrix.
Consequently, we have 
\[
c^{-1}_{NT}\cdot W_{NT}(\tau)\xrightarrow{d} N(0, \Gamma)
\]
where 
\[
\Gamma=\lim_{N,T\to \infty, \tau\to 0 } c^{-2}_{NT}\cdot \Var(W_{NT}).
\]

\end{proof}
\subsection{Auxiliary Lemma: Asymptotic Behavior of $G_{NT}(z,\tau)$ (Quantile Regresion case)} \label{sec:lemma:regression:g}

\begin{lemma}\label{G2}

\[
G_{NT}(z, \tau)\xrightarrow{p} E[G_{NT}(z, \tau)]=  \frac{1}{2} \cdot \frac{m^{-\xi}-1}{-\xi} \cdot z^{\prime} \mathcal{Q}_H z
\]

where 
\[\mathcal{Q}_H \equiv E[[H(X)]^{-1} X X^{\prime}]\],
\[
H(x) \equiv x^{\prime} \mathbf{c} \quad \text { for type } 2 \text { and } 3 \text { tails, } \quad H(x) \equiv 1 \quad \text { for type } 1 \text { tails. }
\]

\end{lemma}

\begin{proof}

Let $F_{it}, f_{it}$ and $E_{it}$ denote $F_U\left(\cdot \mid X_{it}\right), f_U\left(\cdot \mid X_{it}\right)$ and $E\left[\cdot \mid X_{it}\right]$, respectively, where $U$ is the auxiliary error constructed in Assumption \ref{line}.
From Assumption \ref{line},
$x^{\prime}\left(\beta(\tau)-\beta_r\right) \equiv F_U^{-1}(\tau \mid x)$ 
$
\mu_X^{\prime}\left(\beta(\tau)-\beta_r\right) \equiv F_U^{-1}\left(\tau \mid \mu_X\right) \equiv F_u^{-1}(\tau)$.
For computing $E[G_{NT}(z, \tau)]$,  we follow the proof outlined in step 2 of Theorem 5.1 in Chernozhukov (2005).

\begin{equation}
\begin{aligned}
G_{NT}(z, \tau)&\equiv \sum_{i=1}^N \sum_{t=1}^T \frac{a_{NT}}{c^2_{NT}} \cdot\left(\int_0^{c_{NT}\cdot X_{it}^{\prime} z / a_{NT}}\left[\frac{\mathbbm{1}\left(Y_{it}-X_{it}^{\prime} \beta(\tau) \leq s\right)-\mathbbm{1}\left(Y_{it}-X_{it}^{\prime} \beta(\tau) \leq 0\right)}{\sqrt{\tau NT}}\right] d s\right) \\
&=\frac{1}{c_{NT}}\sum_{i=1}^N \sum_{t=1}^T\left(\int_0^{X_{it}^{\prime} z}\left[\frac{\mathbbm{1}\left(Y_{it}-X_{it}^{\prime} \beta(\tau) \leq c_{NT}\cdot s / a_{NT}\right)-\mathbbm{1}\left(Y_{it}-X_{it}^{\prime} \beta(\tau) \leq 0\right)}{\sqrt{\tau NT}}\right] d s\right),
\end{aligned}
\end{equation}

\begin{equation}{\label {c2}}
\begin{aligned}
E G_{NT}(z, \tau) &=\frac{NT}{c_{NT}} \cdot E\left[\int_0^{X_{it}^{\prime} z} \frac{F_{it}\left[F_{it}^{-1}(\tau)+c_{NT}\cdot s / a_{NT}\right]-F_{it}\left[F_{it}^{-1}(\tau)\right]}{\sqrt{\tau NT}} d s\right] \\
&=NT \cdot E\left[\int_0^{X_{it}^{\prime} z} \frac{F_{it}\left[F_{it}^{-1}(\tau)+s / a_{NT}\right]-F_{it}\left[F_{it}^{-1}(\tau)\right]}{\sqrt{\tau NT}} d s\right] \\
& \stackrel{(1)}{=} NT \cdot E\left[\int_0^{X_{it}^{\prime} z} \frac{f_{it}\left\{F_{it}^{-1}(\tau)+o\left(F_u^{-1}(m \tau)-F_u^{-1}(\tau)\right)\right\}}{a_{NT} \cdot \sqrt{\tau NT}} \cdot s \cdot d s\right] \\
& \stackrel{(2)}{\sim} NT \cdot E\left[\int_0^{X_{it}^{\prime} z} \frac{f_{it}\left\{F_{it}^{-1}(\tau)\right\}}{a_{NT} \cdot \sqrt{\tau NT}} \cdot s \cdot d s\right] \\
&=NT \cdot E\left[\frac{1}{2} \cdot\left(X_{it}^{\prime} z\right)^2 \cdot \frac{f_{it}\left\{F_{it}^{-1}(\tau)\right\}}{a_{NT} \cdot \sqrt{\tau NT}}\right] \\
&=E\left[\frac{1}{2} \cdot\left(X_{it}^{\prime} z\right)^2 \cdot \frac{F_u^{-1}(m \tau)-F_u^{-1}(\tau)}{\tau\left(f_{it}\left\{F_{it}^{-1}(\tau)\right\}\right)^{-1}}\right] \\
& \stackrel{(3)}{\sim} E\left[\frac{1}{2} \cdot\left(X_{it}^{\prime} z\right)^2 \cdot \frac{1}{H(X)} \cdot \frac{m^{-\xi}-1}{-\xi}\right] \\
& \equiv \frac{1}{2} \cdot \frac{m^{-\xi}-1}{-\xi} \cdot z^{\prime} \mathcal{Q}_H z .
\end{aligned}
\end{equation}
Equality (1) is by the definition of $a_{NT}$ and a Taylor expansion. Indeed, since $\tau N T \rightarrow \infty$ uniformly over $s$ in any compact subset of $\mathbb{R}$,
\[
s / a_T=s \cdot\left(F_u^{-1}(m \tau)-F_u^{-1}(\tau)\right) / \sqrt{\tau NT}=o\left(F_u^{-1}(m \tau)-F_u^{-1}(\tau)\right).
\]
To show equivalence (2), it suffices to prove that, for any sequence $v_\tau=$ $o\left(F_u^{-1}(m \tau)-F_u^{-1}(\tau)\right)$ with $m>1$ as $\tau \to 0$,
$f_{it}\left(F_{it}^{-1}(\tau)+v_\tau\right) \sim f_{it}\left(F_{it}^{-1}(\tau)\right) \quad \text { uniformly in } i \text{ and } t \text {. }$
From Assumption \ref{regular2},
\begin{equation}\label{(a)}
f_u\left(F_u^{-1}(l \tau / K)\right) \sim(l / K)^{\xi+1} f_u\left(F_u^{-1}(\tau)\right) \sim(l)^{\xi+1} f_u\left(F_u^{-1}(\tau / K)\right)
\end{equation}

Locally uniformly in $l \text { and uniformly in } K \in\{K(x): x \in \mathbf{X}\}$.
Also from Assumption \ref{regular2}
\begin{equation}\label{(b)}
1 / f_{it}\left(F_{it}^{-1}(\tau)\right) \sim \partial F_u^{-1}\left(\tau / K\left(X_{it}\right)\right) / \partial \tau=1 /\left\{K\left(X_{it}\right) f_u\left[F_u^{-1}\left(\tau / K\left(X_{it}\right)\right)\right]\right\}
\end{equation}
uniformly in $i$, $t$.
Combining (\ref{(a)}) and (\ref{(b)}) we have that locally uniform in $l$ and uniformly in $i$ and $t$ ,
\[
f_{it}\left(F_{it}^{-1}(l \tau)\right) \sim l^{\xi+1} f_{it}\left(F_{it}^{-1}(\tau)\right) .
\]
Locally uniform in $l$
\[
f_{it}\left(F_{it}^{-1}(\tau)+\left[F_{it}^{-1}(l \tau)-F_{it}^{-1}(\tau)\right]\right) \sim l^{\xi+1} f_u\left(F_{it}^{-1}(\tau)\right).
\]
Hence, for any $l_\tau \rightarrow 1$,
\[
f_{it}\left(F_{it}^{-1}(\tau)+\left[F_{it}^{-1}\left(l_\tau \tau\right)-F_{it}^{-1}(\tau)\right]\right) \sim f_{it}\left(F_{it}^{-1}(\tau)\right) .
\]
Hence, for any sequence $v_\tau=o\left(\left[F^{-1}(m \tau)-F^{-1}(\tau)\right]\right)$ with $m>1$ as $\tau \to 0$,
\[
f_{it}\left(F^{-1}_{it}(\tau)+v_\tau\right) \sim f_{it}\left(F_{it}^{-1}(\tau)\right)
\]
The equivalence (3) can be shown as 
\[
\frac{F_u^{-1}(m \tau)-F_u^{-1}(\tau)}{\tau\left(f_{it}\left[F_{it}^{-1}(\tau)\right]\right)^{-1}} \sim \frac{F_u^{-1}(m \tau)-F_u^{-1}(\tau)}{\tau\left(K\left(X_{it}\right) f_u\left[F_u^{-1}\left(\tau / K\left(X_{it}\right)\right)\right]\right)^{-1}} .
\]
 By (\ref{(b)}) we have that uniformly in $i$ and $t$
\[
\begin{aligned}
\frac{F_u^{-1}(m \tau)-F_u^{-1}(\tau)}{\tau\left(f_{it}\left[F_{it}^{-1}(\tau)\right]\right)^{-1}} & \sim \frac{1}{K\left(X_{it}\right)^{\xi}} \cdot \frac{F_u^{-1}(m \tau)-F_u^{-1}(\tau)}{\tau\left(f_u\left[F_u^{-1}(\tau)\right]\right)^{-1}} \\
& =\frac{1}{H\left(X_{it}\right)} \cdot \frac{F_u^{-1}(m \tau)-F_u^{-1}(\tau)}{\tau\left(f_u\left[F_u^{-1}(\tau)\right]\right)^{-1}},
\end{aligned}
\]
where $H\left(X_{it}\right)=X_{it}^{\prime} \mathbf{c}$ for $\xi \neq 0$ and $H\left(X_{it}\right)=1$ for $\xi=0$. Finally, by the regular variation property, 
\[
\begin{aligned}
\frac{F_u^{-1}(m \tau)-F_u^{-1}(\tau)}{\tau\left(f_u\left[F_u^{-1}(\tau)\right]\right)^{-1}} & \equiv \int_1^m \frac{f_u\left[F_u^{-1}(\tau)\right]}{f_u\left[F_u^{-1}(s \tau)\right]} d s \\
& \sim \int_1^m s^{-\xi-1} d s \\
& =\frac{m^{-\xi}-1}{-\xi} \quad(\ln m \text { if } \xi=0) .
\end{aligned}
\]
We still have to show that, 
\[
G_{NT}(z, \tau)\xrightarrow{p}  E[G_{NT}(z, \tau)].
\]
Let
\begin{equation*}
\lambda_{it}=\int_0^{X_{it}^{\prime} z}\left[\mathbbm{1}\left(Y_{it}-X_{it}^{\prime} \beta(\tau) \leq s / a_{NT}\right)-\mathbbm{1}\left(Y_{it}-X_{it}^{\prime} \beta(\tau) \leq 0\right)\right] d s .
\end{equation*}

\begin{equation*}
\begin{aligned}
 G_{NT}(z, \tau)=\frac{1}{\sqrt{NT\tau}}\sum_{i=1}^N\sum_{t=1}^T \lambda_{it} =&\frac{-\sqrt{T}}{\sqrt{\tau N}}\sum_{i=1}^N E[\lambda_{it}|\alpha_i]\\
&+\frac{-\sqrt{N}}{\sqrt{\tau T}}\sum_{t=1}^T E[\lambda_{it}|\gamma_{t}]\\
&+\frac{-1}{\sqrt{\tau N T}}\sum_{i=1}^N \sum_{t=1}^T (E[\lambda_{it}|\alpha_i, \gamma_t]-E[\lambda_{it}|\alpha_i]-E[\lambda_{it}|\gamma_t])\\
&+\frac{-1}{\sqrt{\tau NT}}\sum_{i=1}^N \sum_{t=1}^T(\lambda_{it}-E[\lambda_{it}|\alpha_i, \gamma_t])\\
&=B_1+B_2+B_3+B_4.
\end{aligned}
\end{equation*}
Observe that
\begin{equation*}
\begin{aligned}
\frac{\sqrt{NT}}{\sqrt{\tau}}\cdot E[\lambda_{it}|\alpha_i] &=\frac{\sqrt{NT}}{\sqrt{\tau}}\cdot E\left[\int_0^{X_{it}^{\prime} z}\left[\mathbbm{1}\left(Y_{it}-X_{it}^{\prime} \beta(\tau) \leq s / a_{NT}\right)-\mathbbm{1}\left(Y_{it}-X_{it}^{\prime} \beta(\tau) \leq 0\right)\right] d s\middle | \alpha_i\right]\\
& \stackrel{(1)}{=} \frac{\sqrt{NT}}{\sqrt{\tau}} \cdot E\left[\int_0^{X_{it}^{\prime} z} \frac{f_{it}\left\{F_{it}^{-1}(\tau)+o\left(F_u^{-1}(m \tau)-F_u^{-1}(\tau)\right)\right|\alpha_i\}}{a_{NT}} \cdot s \cdot d s\middle | \alpha_i \right] \\
& \stackrel{(2)}{\sim} \frac{\sqrt{NT}}{\sqrt{\tau}}\cdot E\left[\int_0^{X_{it}^{\prime} z} \frac{f_{it}\left\{F_{it}^{-1}(\tau)\right|\alpha_i\}}{a_{NT} } \cdot s \cdot d s \middle | \alpha_i \right] \\
&=\frac{\sqrt{NT}}{\sqrt{\tau}}\cdot E\left[\frac{1}{2} \cdot\left(X_{it}^{\prime} z\right)^2 \cdot \frac{f_{it}\left\{F_{it}^{-1}(\tau)\right|\alpha_i\}}{a_{NT} }\middle | \alpha_i \right] \\
&=E\left[\frac{1}{2} \cdot\left(X_{it}^{\prime} z\right)^2 \cdot \frac{F_u^{-1}(m \tau)-F_u^{-1}(\tau)}{\tau\left(f_{it}\left\{F_{it}^{-1}(\tau)\right|\alpha_i\}\right)^{-1}}\middle | \alpha_i \right] \\
& \stackrel{(3)}{\sim} E\left[\frac{1}{2} \cdot\left(X_{it}^{\prime} z\right)^2 \cdot \frac{1}{H(X_{it})} \cdot \frac{m^{-\xi}-1}{-\xi}\middle | \alpha_i \right] \\
& \equiv \frac{1}{2} \cdot \frac{m^{-\xi}-1}{-\xi} \cdot z^{\prime} E[[H(X_{it})]^{-1} X_{it} X_{it}^\prime | \alpha_i] z .
\end{aligned}
\end{equation*}
Note that (1) and (2) follow from Assumption \ref{conditional regular} (ii) and the same arguments made in equation (\ref{c2}).
(3) follows from Assumption \ref{conditional regular} (ii) and Assumption \ref{density2} (i) as 

\[
\begin{aligned}
\frac{F_u^{-1}(m \tau)-F_u^{-1}(\tau)}{\tau\left(f_u\left[F_u^{-1}(\tau)| \alpha_i\right]\right)^{-1}} & \equiv \int_1^m \frac{f_u\left[F_u^{-1}(\tau)\right]}{f_u\left[F_u^{-1}(s \tau)\right]} \cdot \frac{f_u\left[F_u^{-1}(\tau)| \alpha_i\right]}{f_u\left[F_u^{-1}(\tau)\right]}d s \\
& \sim \int_1^m s^{-\xi-1} d s \\
& =\frac{m^{-\xi}-1}{-\xi} \quad(\ln m \text { if } \xi=0) .
\end{aligned}
\]
It follows that
\[
\Var(E[\lambda_{it}|\alpha_i])=O(E[E[\lambda_{it}|\alpha_i]^2])=
O(\tau/NT)
\]
\[
\Var(B_1)=O(\tau/NT)\cdot \frac{T}{\tau N}\cdot N =O(1/{N}).
\]
Similarly, we have 
\[
\Var(E[\lambda_{it}|\gamma_t])=O(E[E[\lambda_{it}|\gamma_t]^2])=
O(\tau/NT)
\]
\[
\Var(E[\lambda_{it}|\alpha_i,\gamma_t])=O(E[E[\lambda_{it}|\alpha_i,\gamma_t]^2])=
O(\tau/NT).
\]
Hence
\[
\Var(B_2)=O(\tau/NT)\cdot \frac{N}{\tau T}\cdot T =O(1/{T})
\]
\[
\Var(B_3)=O(\tau/NT)\cdot \frac{1}{\tau N T}\cdot T =O(1/{NT}).
\]
Recall that 
\begin{equation*}
\lambda_{it}=\int_0^{X_{it}^{\prime} z}\left[\mathbbm{1}\left(Y_{it}-X_{it}^{\prime} \beta(\tau) \leq s / a_{NT}\right)-\mathbbm{1}\left(Y_{it}-X_{it}^{\prime} \beta(\tau) \leq 0\right)\right] d s .
\end{equation*}
Let
\begin{equation*}
\mu_{it}=\left(\mathbbm{1}\left(Y_{it}-X_{it}^{\prime} \beta(\tau) \leq X_{it}^{\prime} z / a_{NT}\right)-\mathbbm{1}\left(Y_{it}-X_{it}^{\prime} \beta(\tau) \leq 0\right)\right).
\end{equation*}
Then from Assumption \ref{x},
\begin{equation*}
\left|\lambda_{it}\right| \leq K_0\left|\mu_{it}\right|
\end{equation*}
for some $K_0<\infty$. 
\[
\Var(\lambda_{it})=O(E[\lambda^2_{it}])=O(E[\mu^2_{it}])=O(E[\abs{\mu_{it}}])
=O(\sqrt{\tau/NT}).
\]
Last equation follows from equation (\ref{c2}).
Therefore, 
\[\Var(G_{NT}(z, \tau))=\left(O\left(\frac{1}{N}\right)+O\left(\frac{1}{T}\right)+O\left(\frac{1}{NT}\right)+O\left(\frac{1}{\sqrt{NT\tau}}\right)\right).
\]

\[
G_{NT}(z, \tau)\xrightarrow{p} E[G_{NT}(z, \tau)]=  \frac{1}{2} \cdot \frac{m^{-\xi}-1}{-\xi} \cdot z^{\prime} \mathcal{Q}_H z.
\]
\end{proof}

\end{document}